\documentclass[12pt]{article}
\usepackage{amssymb,amsmath,amsthm,amsxtra,amsfonts}
\usepackage{geometry}
\usepackage{aliascnt}
\usepackage{color}
\usepackage[usenames,dvipsnames,table]{xcolor}
\usepackage{subcaption}
\usepackage{bbm}
\usepackage{dsfont}
\usepackage{graphicx}
\usepackage{mathrsfs}
\usepackage{tikz}
\usetikzlibrary{plotmarks,arrows,calc,patterns}
\usepackage{pgfplots}
\usepackage[longnamesfirst,round]{natbib}
\usepackage[colorlinks=true,citecolor=Blue,urlcolor=Blue,linkcolor=Blue]{hyperref}
\usepackage{threeparttable}
\usepackage{booktabs}
\usepackage{doi}

\newtheorem{theorem}{Theorem}[section]
\newtheorem{proposition}{Proposition}[section]
\newtheorem{lemma}{Lemma}[section]
\newtheorem{corollary}{Corollary}[section]
\newtheorem{example}{Example}[section]

\newtheorem{assumption}{Assumption}[section]
\newtheorem{definition}{Definition}[section]

\newcommand{\R}{\ensuremath{\mathbf{R}}}

\newcommand{\Nn}{\ensuremath{\mathbf{N}}}

\newcommand{\E}{\ensuremath{\mathds{E}}}
\newcommand{\one}{\ensuremath{\mathds{1}}}
\newcommand{\dx}{\ensuremath{\mathrm{d}}}
\newcommand{\Var}{\ensuremath{\mathrm{Var}}}
\newcommand{\Cov}{\ensuremath{\mathrm{Cov}}}

\usepackage{xr}
\begin{document}

\title{Is Completeness Necessary? Estimation in Nonidentified Linear Models\footnote{\scriptsize We are grateful to the { Editor, Co-Editor, and} three anonymous referees for helpful comments. We are also grateful to the participants of the Duke workshop, TSE Econometrics seminar, Triangle Econometrics Conference, 4th ISNPS Conference, 2018 NASMES Conference, and Bristol Econometric Study Group. Jean-Pierre Florens acknowledges funding from the French National Research Agency (ANR) under the Investments for the Future program (Investissement d'Avenir, grant ANR-17-EURE-0010). All remaining errors are ours.}}
\author{
	Andrii Babii\footnote{University of North Carolina at Chapel Hill - Gardner Hall, CB 3305
		Chapel Hill, NC 27599-3305. Email: \href{mailto:babii.andrii@gmail.com}{babii.andrii@gmail.com}.} \\
	\textit{\normalsize UNC Chapel Hill}
	\and
	Jean-Pierre Florens\footnote{Toulouse School of Economics -- 1, Esplanade de l'Universit\'{e}, 31080 Toulouse, France. Email: \href{mailto:jean-pierre.florens@tse-fr.eu}{jean-pierre.florens@tse-fr.eu}} \\
	\textit{\normalsize Toulouse School of Economics}
}
\maketitle

\begin{abstract}
	{\footnotesize Modern data analysis depends increasingly on estimating models via flexible high-dimensional or nonparametric machine learning methods, where the identification of structural parameters is often challenging and untestable. In linear settings, this identification hinges on the completeness condition, which requires the nonsingularity of a high-dimensional matrix or operator and may fail for finite samples or even at the population level. Regularized estimators provide a solution by enabling consistent estimation of structural or average structural functions, sometimes even under identification failure. We show that the asymptotic distribution in these cases can be nonstandard. We develop a comprehensive theory of regularized estimators, which include methods such as high-dimensional ridge regularization, gradient descent, and principal component analysis (PCA). The results are illustrated for high-dimensional and nonparametric instrumental variable regressions and are supported through simulation experiments.}
\end{abstract}

\noindent {\footnotesize \textbf{Keywords}: machine learning, high-dimensional regressions, ridge, gradient descent, PCA, nonparametric IV, nonidentified models, weak identification. }

\thispagestyle{empty}


\setcounter{page}{1} 

\section{Introduction}
Structural nonparametric and high-dimensional models are often ill-posed. Notable examples include nonparametric instrumental variable (IV) regression, various high-dimensional regressions, measurement error models, and random coefficient models. These cases commonly lead to an ill-posed functional equation
\begin{equation*}
	K\varphi = r,
\end{equation*}
where $\varphi$ is the structural parameter of interest, $r$ is a known function, and $K$ is a linear operator. The classical literature on numerical ill-posed inverse problems (see \cite{engl1996regularization}) addresses \textit{deterministic} problems, where $K$ is known and where $r$ is observed with a deterministic numerical error. However, in econometric applications, both $K$ and $r$ must be estimated from data, posing a statistically ill-posed inverse problem.

Identification has long been a central concept in econometrics, as demonstrated by foundational works such as \cite{koopmans1949identification}, \cite{koopmans1950identification}, and \cite{rothenberg1971identification} for parametric models. In nonparametric and high-dimensional settings, $K$ and $r$ are typically derived from the data-generating process, with $\varphi$ uniquely identified if $K\varphi = r$ has a unique solution. Uniqueness is achieved when $K$ is a one-to-one (or nonsingular) operator, meaning that $K\phi = 0 \implies \phi = 0$ for all $\phi$ in the domain of $K$. However, in econometric applications, $K$ is often unknown, and the estimated operator $\hat{K}$, being finite-rank, lacks one-to-one properties in finite samples.

In cases of nonidentification, the maximum likelihood estimator can yield a flat likelihood in some parameter space regions, leading to an ambiguously defined maximum. It is thus important to characterize the asymptotic behavior of estimators in potentially nonidentified models. In ill-posed models without identification, the identified set is a linear manifold, $\phi + \mathcal{N}(K)$, where $\phi$ solves $K\phi = r$ and where $\mathcal{N}(K)$ is the null space of $K$. This identified set, which is generally unbounded, lacks informativeness on $\varphi$ without additional constraints.

Given that $K$ or $\hat{K}$ may fail to be one-to-one and often has a discontinuous generalized inverse, regularization is essential for estimating the structural parameter consistently. In this paper, we focus on spectral regularization methods that modify the spectrum of $\hat{K}$. A prominent example is Tikhonov regularization (functional ridge regression); see \cite{tikhonov1963solution}. Other methods include iterated Tikhonov, spectral cutoff (or principal component analysis, see \cite{yao2005functional}), and Landweber iteration (or functional gradient descent).\footnote{Functional gradient descent is also foundational in boosting estimators; see \cite{friedman2001greedy}.} We show that estimators based on spectral regularization are uniformly consistent for the best approximation of the structural parameter in the orthogonal complement of the null space of $K$, and we explore the distributional implications of identification failures.

In some cases, the best approximation may coincide with, or closely approximate, the structural parameter even if $K$ is singular and has a nontrivial null space. This provides an insightful interpretation of nonparametric IV regression under identification failure, akin to the least-squares best approximation under misspecification; see \cite{angrist2008mostly}, Chapter 3.\footnote{In contrast, the parametric 2SLS estimator lacks this best approximation property; see \cite{escanciano2018optimal} for an example of an IV estimator with this interpretation.} Moreover, the best approximation can also be used to construct set estimators in partial identification frameworks.

\paragraph{Contribution and related literature.} Our paper contributes to several research areas. First, it advances the literature on regularization in statistical inverse problems, as discussed in \cite{newey2003instrumental}, \cite{blundell2007semi}, \cite{chen2012estimation}, \cite{chen2018optimal}, \cite{carrasco2007linear,carrasco2013asymptotic}, \cite{darolles2011nonparametric}, \cite{florens2011identification}, \cite{gagliardini2012tikhonov}, \cite{hall2005nonparametric}, \cite{babii2020honest,babii2021high}, and other studies. \cite{canay2013testability} demonstrate that the completeness condition is untestable in these models; see also \cite{d2010completeness}, \cite{hu2017nonparametric}, and \cite{andrews2017examples} for more on this condition. Our contribution is the derivation of convergence rates in Hilbert space and $L_\infty$ norms for estimators based on generic spectral regularization schemes, including Tikhonov, iterated Tikhonov, Landweber iteration, and spectral cutoff. These results, which do not require the completeness condition, are illustrated for nonparametric IV and high-dimensional regressions.\footnote{This generalizes \cite{florens2011identification}, which covers only $L_2$ rates for Tikhonov regularization under certain high-level conditions.}

Second, we add to the literature on identifying linear functionals of a structural function $\varphi$. For instance, \cite{severini2006some,severini2012efficiency} show that linear functionals of $\varphi$ can be identified under conditions weaker than the completeness condition and derive the corresponding efficiency bounds when $\sqrt{n}$ estimation is possible. \cite{santos2011instrumental} establishes the asymptotic normality of linear functionals of nonparametric IV regression without the completeness condition, whereas \cite{escanciano2018optimal} provides necessary and sufficient conditions for the identification and $\sqrt{n}$-estimability of the best linear approximation to the nonparametric IV structural function. We contribute by showing that the asymptotic distribution of linear functionals of $\varphi$ may include a chi-square component when the completeness condition fails and the functional is estimable at a rate slower than $\sqrt{n}$.

Third, our work connects to the partial identification literature in nonparametric IV regression; see \cite{santos2012inference}, \cite{freyberger2017completeness}, and \cite{chernozhukov2023constrained}. Our findings on consistency for the best approximation in the $L_\infty$ norm, alongside \cite{freyberger2017completeness}, have applications in partial identification.

The remainder of this paper is structured as follows. Section~\ref{sec:geometry} introduces the concept of identification in linear ill-posed models. Section~\ref{sec:rates} derives nonasymptotic risk bounds in Hilbert space and $L_\infty$ norms for Tikhonov-regularized estimators, generalizing these results to other spectral regularization schemes in Appendix Section~\ref{sec:general}. Section~\ref{sec:flir} illustrates the transition between Gaussian and weighted chi-square asymptotics under identification failure. Section~\ref{sec:mc} presents a Monte Carlo study, while Section~\ref{sec:conclusion} concludes the paper. All proofs are presented in Appendix~\ref{sec:pi}. Appendix~\ref{sec:extreme} discusses partial identification, and Appendix~\ref{sec:extreme} examines Tikhonov-regularized estimators under extreme identification failures. Finally, Online Supplementary Material, Sections~\ref{app:gi} and \ref{app:wiener} review relevant results from generalized inverse theory and Hilbert space-valued U-statistics.

\section{Identification}\label{sec:geometry}
Consider the functional linear equation
\begin{equation*}
	K\varphi = r,
\end{equation*}
where $K:\mathcal{E}\to\mathcal{H}$ is a compact linear operator between Hilbert spaces $(\mathcal{E},\langle.,.\rangle_{\mathcal{E}})$ and $(\mathcal{H},\langle.,.\rangle_{\mathcal{H}})$ and where $\varphi\in\mathcal{E}$ represents the structural parameter of interest. For $\varphi$ to be point identified, $K$ must be one-to-one (or nonsingular), meaning that the null space of $K$, denoted by $\mathcal{N}(K) = \left\{\phi\in\mathcal{E}:\; K\phi = 0\right\}$, is reduced to $\{0\}$. Thus, for any $\phi\in\mathcal{E}$,
\begin{equation*}
	K\phi = 0\implies \phi = 0.
\end{equation*}

When the completeness condition fails, the identified set becomes a closed linear manifold, which is denoted by $I_0 = \varphi + \mathcal{N}(K)$. Since $\mathcal{N}(K)$ forms a closed linear subspace in $\mathcal{E}$, we can decompose $\varphi = \varphi_1 + \varphi_0$, where $\varphi_0$ is the unique orthogonal projection of $\varphi$ on $\mathcal{N}(K)$ and where $\varphi_1$ is the projection of $\varphi$ on its orthogonal complement, denoted by $\mathcal{N}(K)^\perp$. We let $K^*:\mathcal{H}\to\mathcal{E}$ denote the Hilbert adjoint operator of $K:\mathcal{E}\to\mathcal{H}$, defined by the relation\footnote{To reduce visual clutter, we omit subscripts and use $\langle .,.\rangle$ and $\|.\|$ to denote the inner products and norms in Hilbert spaces. For example, it is clear that the inner product on the left side is $\langle.,.\rangle_\mathcal{H}$, whereas the inner product on the right side is $\langle.,.\rangle_{\mathcal{E}}$.}
\begin{equation*}
	\langle \psi,K\phi\rangle = \langle K^*\psi,\phi\rangle,\qquad \forall (\psi,\phi)\in\mathcal{H}\times\mathcal{E}.
\end{equation*}

Since $\mathcal{N}(K)^\perp=\overline{\mathcal{R}(K^*)}$ (see \cite{luenberger1997optimization}, p.157), the best approximation $\varphi_1$ equals the structural parameter $\varphi$ whenever $\varphi\in \overline{\mathcal{R}(K^*)}$, which is the closure of the range of the adjoint operator $K$. Note that $\mathcal{R}(K^*) = \mathcal{R}(K^*K)^{1/2}$, as shown in \cite{engl1996regularization}, Proposition 2.18. Therefore, the identification of $\varphi$ can be related to a regularity condition known as the source condition, where $\varphi\in\mathcal{R}(K^*K)^{\beta/2}$ with $\beta\geq 1$; see \cite{carrasco2007linear}. In other words, additional regularity of $\varphi$ can compensate for the failure of the completeness condition.\footnote{A similar observation holds for high-dimensional sparse regressions, where sparse coefficients can be identified under the restricted eigenvalue condition; see \cite{tsybakov2009simultaneous}.} The following two examples further illustrate this point; see also \cite{bontemps2024functional} for further applications to ecological inference.

\begin{example}[Nonparametric IV regression]\label{ex:id}
	Consider
	\begin{equation*}
		Y = \varphi(Z) + U,\qquad \E[U|W] = 0,
	\end{equation*}
	where $(Y,Z,W)\in\R\times\R^p\times\R^q$. The exclusion restriction leads to
	\begin{equation*}
		r(w) \triangleq \E[Y|W=w] = \E[\varphi(Z)|W=w] \triangleq (K\varphi)(w),
	\end{equation*}
	where $K:L_2(Z)\to L_2(W)$ is a conditional expectation operator, and for a random vector $X$, we define $L_2(X)=\{\phi:\E|\phi(X)|^2<\infty \}$. If $K$ is compact, then according to the spectral theorem, there {exists} a sequence $(\sigma_j,e_j,h_j)_{j\geq 1}$, where $\sigma_j\to0$ are singular values, $(e_j)_{j\geq1}$ is a complete orthonormal system of $\mathcal{N}(K)^\perp=\overline{\mathcal{R}(K^*)}$, and $(h_j)_{j\geq1}$ is the complete orthonormal system of $\mathcal{N}^\perp(K^*)=\overline{\mathcal{R}(K)}$; see \cite{kress2013linear}, Theorem 15.16. $\varphi_1$ is the same as $\varphi$ whenever it can be represented in terms of $(e_j)_{j\geq 1}$.
\end{example}

\begin{example}[Nonparametric IV with a discrete instrumental variable]
	Following the previous example, we suppose that the instrumental variable is discrete, namely, $W\in\{w_k:\; k\geq 1\}$. Substituting $f_k(z)=f_{Z|W=w_k}(z)$ for every $k\geq 1$ yields
	\begin{equation*}
		\mathcal{N}(K) = \left\{\phi\in L_2(Z):\: \langle \phi,f_k\rangle = 0,\;\forall k\geq 1 \right\},
	\end{equation*}
	where $\langle \phi,f_k\rangle = \E[\phi(Z)f_k(Z)]$. If $\varphi\in \mathrm{span}\{f_k:\;k\geq 1\}$, then $\varphi\in\mathcal{N}(K)^\perp$. $\varphi_1$ matches $\varphi$ whenever it can be expressed in terms of $(f_k)_{k\geq 1}$.
\end{example}

\section{Nonasymptotic Risk Bounds}\label{sec:rates}
In this section, we derive risk bounds for the Tikhonov estimator, which solves
\begin{equation*}
	\min_{\phi} \left\|\hat K\phi - \hat r\right\|^2 + \alpha_n\|\phi\|^2,
\end{equation*}
where $\alpha_n > 0$ is a regularization parameter. The solution to this problem is 
\begin{equation*}
	\hat{\varphi} = (\alpha_n I + \hat K^*\hat K)^{-1}\hat K^*\hat r,
\end{equation*}
which is straightforward to verify.
Note that the estimator $\hat{\varphi}$ is well defined even if $K$ or $\hat{K}$ is not one-to-one.

\subsection{Hilbert Space Risk}
We first describe the relevant class of structural functions and operators:
\begin{assumption}\label{as:source4}
	The structural parameter $\varphi=\varphi_1+\varphi_0$ and the operator $K$ belong to
	\begin{equation*}
		\mathcal{F} = \mathcal{F}(\beta,C)=\left\{(\varphi,K):\; \varphi_1=(K^*K)^{\beta/2}\psi,\;\|\psi\|^2\vee\|\varphi_0\|\vee\|K\|\leq C\right\}
	\end{equation*}
	for some $\beta,C>0$, where $\|K\|=\sup_{\|\phi\|\leq 1}\|K\phi\|$ and $a\vee b=\max\{a,b\}$.
\end{assumption}
To illustrate this assumption, we let $(\sigma_j,e_j,h_j)_{j\geq 1}$ be the spectral decomposition of $K:\mathcal{E}\to\mathcal{H}$; see \cite{kress2013linear}, Theorem 15.16. Then, $\varphi_1 = \sum_{j\geq 1}\langle \varphi_1,e_j\rangle e_j$, and by Parseval's identity, since $\psi = (K^*K)^{-\beta/2}\varphi_1$, we have
\begin{equation*}
	\|\psi\|^2 = \sum_{j\geq 1}\frac{|\langle\varphi_1, e_j\rangle|^2}{\sigma_j^{2\beta}}.
\end{equation*}
Assumption~\ref{as:source4} restricts the relative rates of decline of singular values $(\sigma_j)_{j\geq 1}$, which describe the ill-posedness, and Fourier coefficients $\langle\varphi_1,e_j\rangle_{j\geq 1}$, which describe the regularity of $\varphi_1$.

We estimate $(r,K)$ with $(\hat r,\hat K)$ and impose the following assumption:
\begin{assumption}\label{as:rates4}
	(i) $\E\|\hat r - \hat K\varphi_1\|^2 \leq C_1\delta_{n}$ and (ii) $\E\|\hat K - K\|^2 \leq C_2\rho_{1n},$ where the constants $C_1,C_2<\infty$ do not depend on $(\varphi,K)$ and $\delta_n,\rho_{1n}\to 0$ as $n\to\infty$.
\end{assumption}
Assumption~\ref{as:rates4} restricts the convergence rate of the residuals and the estimated operator. Note that \cite{florens2011identification}, Theorem 2.2, imposes a stronger high-level condition: 
\begin{equation*}
	\E\|\hat K - K\|^4 \leq C\left(\E\|\hat K - K\|^2\right)^2
\end{equation*}
and certain assumptions on the tuning parameters that have not been verified for specific models. 

The following result applies to the Hilbert space norm, $\|.\|$: 
\begin{theorem}\label{thm:L2_rate}
	Suppose that Assumptions~\ref{as:source4} and \ref{as:rates4} hold. Then, for all $\beta\in(0,2]$,
	\begin{equation*}
		\sup_{(\varphi,K)\in\mathcal{F}}\E\left\|\hat{\varphi} - \varphi_1\right\|^2 = O\left(\frac{\delta_{n} + \rho_{1n}\alpha_n^{\beta\wedge 1}}{\alpha_n} + \alpha^{\beta}_n\right),
	\end{equation*}
	where $a\wedge b = \min\{a,b\}$.
\end{theorem}
Theorem~\ref{thm:L2_rate} shows that the convergence rate in the Hilbert space norm is driven by the rates of 
\begin{itemize}
	\item the residuals, $O\left(\delta_n/\alpha_n\right)$;
	\item the estimated operator, $O\left(\rho_{1n}\alpha_n^{\beta\wedge 1}/\alpha_n\right)$;
	\item the regularization bias, $O\left(\alpha_n^{\beta}\right)$.
\end{itemize}
Notably, a stronger version of Assumption~\ref{as:source4} is commonly assumed: $\varphi\in\mathcal{R}(K^*K)^{\beta/2}$ with $\beta\geq 1$; see \cite{carrasco2007linear}. In this case, we actually have $\varphi_0=0$, and the Tikhonov regularized estimator can consistently estimate the structural parameter $\varphi$, despite the failure of the completeness condition. More generally, we distinguish the following possibilities:
\begin{itemize}
	\item the identified case: $\varphi_0=0$ and $\rho_{1n}\alpha_n^{\beta\wedge 1}\lesssim \delta_{n}$, where the rate of convergence to $\varphi$ is driven by the residuals;
	\item the weakly identified case: $\varphi_0=0$ and $\delta_{n}\lesssim \rho_{1n}\alpha_n^{\beta\wedge 1}$, where the rate of convergence to $\varphi$ is driven by the estimated operator;
	\item the nonidentified case: $\varphi_0\ne 0$ and $\rho_{1n}\alpha_n^{\beta\wedge 1}\lesssim \delta_{n}$, where the rate of convergence to $\varphi_1$ is driven by the residuals;
	\item strongly nonidentified models: $\varphi_0\ne 0$ and $\delta_{n}\lesssim \rho_{1n}\alpha_n^{\beta\wedge 1}$, where the rate of convergence to $\varphi_1$ is driven by the estimated operator.
\end{itemize}

In the identified case, the optimal choice of the regularization parameter is $\alpha_n \sim \delta_n^{1/(\beta+1)}$, which yields a convergence rate of $O\left(\delta_n^{\beta/(\beta+1)}\right)$ for $\beta \in (0,2]$. For $\beta > 2$, we demonstrate in Appendix Section~\ref{sec:general} that the optimal rate can be achieved with additional Tikhonov iterations or alternative regularization schemes.

Finally, it is worth mentioning that in the special case of $L_2$ spaces, Theorem~\ref{thm:L2_rate} provides a sharper result than does \cite{florens2011identification}, Theorem 2.2, where the rate is always driven by $\rho_{1n}$.

\subsection{$L_\infty$ Risk}
We consider the space of continuous functions on a compact set $D\subset\R^p$, denoted by $(C(D),\|.\|_\infty)$, where $\|.\|_\infty$ is the uniform norm. We suppose that this space is embedded into $\mathcal{E}$ and that $\mathcal{R}(K^*)\subset C(D)$ and that $\varphi_1\in C(D)$. The mixed operator norm of $K^*$ is defined as $\|K^*\|_{2,\infty}=\sup_{\|\phi\|\leq 1}\|K^*\phi\|_\infty$. The following assumption captures the approximation quality of the operator $K^*$ by its estimate $\hat K^*$ in the $\|.\|_{2,\infty}$ norm.

\begin{assumption}\label{as:rates_2-infty}
	We suppose that $\|K^*\|_{2,\infty}\leq C_3$ and $\E\|\hat K^* - K^*\|_{2,\infty}^2 \leq C_3 \rho_{2n}$, where $C_3<\infty$ is independent of $(\varphi,K)$ and $\rho_{2n}\to 0$.
\end{assumption}

The following result establishes a uniform convergence rate:
\begin{theorem}\label{thm:Linf_rate}
	Under Assumptions~\ref{as:source4}, \ref{as:rates4}, and \ref{as:rates_2-infty}, for all $\beta\in(0,2]$,
	\begin{equation*}
		\sup_{(\varphi,K)\in\mathcal{F}}\E\left\|\hat{\varphi} - \varphi_1\right\|_\infty = O\left(\frac{\delta_{n}^{1/2} + \rho_{1n}^{1/2}\alpha_n^{\frac{\beta+1}{2}\wedge 1} + \rho_{2n}^{1/2}\alpha_n^{1/2}}{\alpha_n} + \alpha_n^{\beta/2}\right).
	\end{equation*}
\end{theorem}
Theorem~\ref{thm:Linf_rate} provides a uniform convergence rate for generic ill-posed inverse problems. Unlike Theorem~\ref{thm:L2_rate}, the $L_\infty$ convergence rate is affected by the estimation error in $K^*$, assessed in the $\|.\|_{2,\infty}$ norm. For further details on the $L_\infty$ convergence rates of more general regularized estimators, see Appendix Section~\ref{sec:general}.

\subsection{Applications}
\subsubsection{High-Dimensional Regressions}\label{sec:flir_example}
We consider the model
\begin{equation*}
	Y = \langle Z,\varphi\rangle + U,\qquad \E[UW]=0,
\end{equation*}
where $(Y,Z,W)\in\R\times\mathcal{E}\times\mathcal{H}$; see \cite{florens2015instrumental}.\footnote{For example, when $W=Z$ and $\mathcal{E}=L_2$ with a counting measure, this yields a high-dimensional regression model for nonsparse data $$Y = \sum_{j\geq 1}\varphi_jZ_j + U,\qquad \E[UZ_j]=0,\qquad \forall j\geq 1.$$ When $\mathcal{E}=L_2$ with the Lebesgue measure, this model is also referred to as functional regression; see \cite{florens2015instrumental} and \cite{babii2021high}.} The exclusion restriction leads to the functional linear equation
\begin{equation*}
	r \triangleq \E[YW] = \E[W\otimes Z]\varphi \triangleq K\varphi,
\end{equation*}
where $K:\mathcal{E}\to\mathcal{H}$ is a covariance operator and where $\otimes$ denotes the tensor product. An econometrician observes an i.i.d. sample $(Y_i,Z_i,W_i)_{i=1}^n$.\footnote{This assumption can be relaxed to covariance stationarity and absolute summability of autocovariances; see \cite{babii2021high}.} Then,
\begin{equation*}
	r = \E[YW],\qquad K = \E[W\otimes Z],\qquad K^*=\E[Z\otimes W]
\end{equation*}
are estimated as
\begin{equation*}
	\hat r = \frac{1}{n}\sum_{i=1}^nY_iW_i,\qquad \hat K = \frac{1}{n}\sum_{i=1}^nW_i\otimes Z_i,\qquad \hat K^* = \frac{1}{n}\sum_{i=1}^nZ_i\otimes W_i.
\end{equation*}

It can be shown that 
\begin{equation*}\small
	\begin{aligned}
		\E\left\|\hat r - \hat K\varphi_1\right\|^2 = \frac{\E\|(Y-\langle Z,\varphi_1\rangle)W\|^2}{n}\qquad\text{and}\qquad \E\left\|\hat K - K\right\|^2 \leq \frac{\E\|ZW\|^2}{n}.
	\end{aligned}
\end{equation*}

We let $\mathcal{F}(\beta,C)$ represent the class of models satisfying Assumption~\ref{as:source4} and assume that $\E\|UW\|^2\vee \E\|ZW\|^2\leq C$ for all the models in this class. Then, $\delta_{n}=\rho_{1n}=n^{-1}$, and by Theorem~\ref{thm:L2_rate}, the Hilbert space risk is of order
\begin{equation*}
	\sup_{(\varphi,K)\in\mathcal{F}}\E\left\|\hat{\varphi} - \varphi_1\right\|^2 = O\left(\frac{1}{\alpha_nn} + \alpha_n^{\beta}\right).
\end{equation*}
For high-dimensional regression, the model is either identified ($\varphi_0=0$) or nonidentified ($\varphi_0\ne 0$). The conditions $\alpha_n\to 0$ and $\alpha_nn\to\infty$ as $n\to\infty$ are sufficient for consistency in the Hilbert space norm. The optimal choice $\alpha_n\sim n^{-1/(\beta+1)}$ yields a convergence rate of order $O(n^{-\beta/(\beta+1)})$, which is optimal when $W=Z$; see \cite{babii2024functional}, Theorem 2.2.

For uniform convergence, we suppose that $\mathcal{E}=L_2(D)$, where $D\subset\R^p$ is bounded, and that the functions in $L_2(D)$ are square integrable with respect to the Lebesgue measure. To satisfy Assumption~\ref{as:rates_2-infty}, we also assume that the models in $\mathcal{F}(\beta,C)$ are such that $\|Z\|_\infty\vee \|W\|_\infty\leq C<\infty$ and that stochastic processes $Z$ and $W$ belong to a H\"{o}lder ball with smoothness $s>d/2$. Using the Hoffman--J\o rgensen and moment inequalities (see \cite{gine2015mathematical}, p. 129 and p. 202),
\begin{equation*}
	\E\left\|\hat K^* - K^*\right\|_{2,\infty}^2 = O\left(\frac{1}{n}\right).
\end{equation*}
Therefore, Assumption~\ref{as:rates_2-infty} holds with $\rho_{2n}=n^{-1}$, and by Theorem~\ref{thm:Linf_rate},
\begin{equation*}
	\sup_{(\varphi,K)\in\mathcal{F}}\E\left\|\hat \varphi - \varphi_1\right\|_\infty = O\left(\frac{1}{\alpha_n n^{1/2}} + \alpha_n^{\beta/2}\right).
\end{equation*}
The conditions $\alpha_n\to 0$ and $\alpha_n n^{1/2}\to \infty$ as $n\to\infty$ ensure the uniform consistency of $\hat\varphi$. The optimal choice $\alpha_n \sim n^{-1/(\beta+2)}$ leads to a convergence rate of order $O(n^{-\beta/{(2\beta+4)}})$. Determining the minimax-optimal convergence rate for the class $\mathcal{F}(\beta,C)$ remains an open problem.

\subsubsection{Nonparametric IV Regression}
Following Example~\ref{ex:id}, we can reformulate the model as 
\begin{equation*}
	r(w) \triangleq \E[Y|W=w]f_W(w) = \int \varphi(z)f_{ZW}(z,w)\dx z \triangleq (K\varphi)(w),
\end{equation*}
where $K:L_2([0,1]^p)\to L_2([0,1]^q)$. To estimate $r$ and $K$, we use kernel smoothing:
\begin{equation*}
	\begin{aligned}
		\hat r(w) & = \frac{1}{nh_n^q}\sum_{i=1}^nY_iK_w\left(h_n^{-1}(W_i-w)\right), \\
		(\hat K\phi)(w) & = \int \phi(z)\hat f_{ZW}(z,w)\dx z, \\
		\hat f_{ZW}(z,w) & = \frac{1}{nh_n^{p+q}}\sum_{i=1}^nK_z\left(h_n^{-1}(Z_i-z)\right)K_w\left(h_n^{-1}(W_i-w)\right),
	\end{aligned}
\end{equation*}
where $K_w$ and $K_z$ are symmetric kernel functions and where $h_n\to 0$ is a bandwidth parameter. Under mild assumptions, Proposition~\ref{prop:npiv_components} shows that $\delta_{n} = \frac{1}{nh^q_n} + h^{2s}_n$ and $\rho_{1n} = \frac{1}{nh^{p+q}_n} + h^{2s}_n$, where $s$ denotes the H\"{o}lder smoothness of $f_{ZW}$. Thus, by Theorem~\ref{thm:L2_rate}, the mean-integrated squared error has the following rate:
\begin{equation*}
	\sup_{(\varphi,K)\in\mathcal{F}}\E\left\|\hat{\varphi} - \varphi_1\right\|^2 = O\left(\frac{1}{\alpha_n}\left(\frac{1}{nh_n^{q}} + h_n^{2s}\right) + \frac{1}{nh_n^{p+q}}\alpha_n^{(\beta-1)\wedge 0} + \alpha_n^{\beta}\right),
\end{equation*}
where the class $\mathcal{F}(\beta,C)$ includes additional moment restrictions; see \cite{babii2020honest}. In the nonparametric IV model, all four identification cases are possible, depending on the regularity parameter $\beta$. For consistency of $\hat\varphi$, we require $\alpha_nnh_n^{q}\to\infty$, $\alpha_n^{(1-\beta)\wedge 0}nh_n^{p+q}\to\infty$, and $h_n^{2s}/\alpha_n\to 0$ as $n\to \infty,\alpha_n\to0$, and $h_n\to0$.

Moreover, it is known that $\rho_{2n}^{1/2} = \sqrt{\frac{\log h_n^{-1}}{nh^{p+q}_n}} + h^s_n$; see, for instance, \cite{babii2020honest}, Proposition A.3.1. Thus, by Theorem~\ref{thm:Linf_rate}, we find
\begin{equation*}
	\sup_{(\varphi,K)\in\mathcal{F}}\E\|\hat \varphi - \varphi_1\|_\infty = O\left(\frac{1}{\alpha_n}\left(\frac{1}{\sqrt{nh_n^{q}}} + h_n^{s}\right) + \frac{1}{\alpha_n^{1/2}} \sqrt{\frac{\log h_n^{-1}}{nh_n^{p+q}}} + \alpha_n^{\beta/2}\right).
\end{equation*}
For uniform consistency of $\hat\varphi$, we need $\alpha_n^2nh_n^q\to \infty$, $\alpha_nnh_n^{p+q}\to\infty$, and $h_n^{s}/\alpha_n\to 0$ as $n\to\infty$, $\alpha_n\to0$, and $h_n\to 0$.

It remains an open question whether these convergence rates for the NPIV model are minimax-optimal for the class $\mathcal{F}(\beta,C)$. Note that this class does not require differentiability or specific H\"{o}lder smoothness of $\varphi_1$; see \cite{chen2018optimal}, which established minimax-optimal convergence rates for the NPIV model over the H\"{o}lder ball.

\section{Continuous Linear Functionals}\label{sec:flir}
In many economic applications, the primary focus is on a continuous linear functional of a structural function $\varphi$. This section demonstrates that the asymptotic distribution of such functionals can become nonstandard under identification failure. To simplify the exposition, we concentrate on the high-dimensional regression example in Section~\ref{sec:flir_example}; further results are provided in Appendix Section~\ref{sec:extreme}.

According to the Riesz representation theorem, any continuous linear functional on a Hilbert space $\mathcal{E}$ can be represented as an inner product with some $\mu\in \mathcal{E}$. Thus, we focus on the asymptotic distribution of $\langle \hat\varphi,\mu\rangle$. We let $\mu = \mu_0+\mu_1$ denote the orthogonal decomposition of $\mu$, where $\mu_0$ and $\mu_1$ are the orthogonal projections on $\mathcal{N}(K)$ and $\mathcal{N}(K)^\perp$, respectively. Similarly, we decompose $W = W^0 + W^1$, where $W^0$ and $W^1$ are the orthogonal projections of $W$ on $\mathcal{N}(K^*)$ and $\mathcal{N}(K^*)^\perp$, respectively.

We define the variance operator $\Sigma=\E[(Y-\langle Z,\varphi_1\rangle)^2W\otimes W]$ and the sequence
\begin{equation*}
	\pi_n = n^{1/2}\left\|\Sigma^{1/2}K(\alpha_n I + K^*K)^{-1}\mu_1\right\|^{-1},
\end{equation*}
with $\mu_1=\mu-\mu_0$. The following assumptions place mild restrictions on the data distribution and tuning parameters.
\begin{assumption}\label{as:data4}
	(i) The data $(Y_i,Z_i,W_i)_{i = 1}^n$ are an i.i.d. sample of $(Y,Z,W)$ with $\E|U|^{2+\delta}<\infty$ and $\E\|Z\|^{2+\delta}<\infty$ for some $\delta>0$; (ii) $\E\|ZW\|^2<\infty$, $\E\|UW\|^2<\infty$, $\E\|UZW\|<\infty$, $\E\|Z\|^2\|W\|<\infty$, and $\E\left[|U|\|Z\|\|W\|^2\right]<\infty$; and (iii) $W^0\ne 0$.
\end{assumption}

\begin{assumption}\label{as:tuning_inner}
	(i) $\mu_1\in\mathcal{R}(K^*K)^{\gamma}$ and $W\in\mathcal{R}(K^*K)^{\tilde\gamma}$ for some $\gamma,\tilde\gamma>0$ with $\tilde \gamma + \gamma\geq 1/2$ and (ii) $\alpha_n\to 0$, $n\alpha_n\to\infty$, $\pi_n\alpha_n^{(\gamma+\beta/2)\wedge 1}\to 0$, $\frac{\pi_n\alpha_n^{\gamma\wedge\frac{1}{2}}}{n\alpha_n} \to 0$, and $n\alpha_n^{1+\beta\wedge 1}\to 0$ as $n\to\infty$.
\end{assumption}
Assumption~\ref{as:tuning_inner} (ii) is an undersmoothing condition. Its requirements are the most stringent when $\pi_n\sim n^{1/2}$, in which case $n\alpha_n^{(2\gamma + \beta)\wedge 2}\to 0$, $n\alpha_n^{2-2\gamma\wedge 1}\to \infty$, and $n\alpha_n^{1+\beta\wedge 1}\to 0$ are needed. The last two conditions are not binding whenever $\gamma\geq1/2$ and $\beta>0$, in which case the first condition is automatically satisfied.

We let $(\chi_{j}^2)_{j\geq 1}$ be i.i.d. chi-square random variables with 1 degree of freedom, and we let $(\lambda_j)_{j\geq 1}$ be the eigenvalues of the operator $T:\phi\mapsto \E_{X}[\phi(X)h(X,X')]$ on $L_2(X)$ with
\begin{equation*}
	h(X,X') = \frac{\langle W^0,W^{0'}\rangle}{2}\left\{\langle Z,\mu_0\rangle (Y' - \langle Z',\varphi_1\rangle) + \langle Z',\mu_0\rangle(Y - \langle Z,\varphi_1\rangle)\right\},
\end{equation*}
where $\E_X$ denotes the expectation with respect to $X=(Y,Z,W)$ and where $X'$ is an independent copy of $X$.

The following result holds:
\begin{theorem}\label{thm:inner_products_flir}
	Under Assumptions~\ref{as:source4}, \ref{as:data4}, and \ref{as:tuning_inner}, we have
	\begin{equation*}
		n\alpha_n\langle \hat\varphi - \varphi_1,\mu_0\rangle \xrightarrow{d} \E\left[\|W\|^2(Y - \langle Z,\varphi_1\rangle)\langle Z,\mu_0\rangle\right] + \sum_{j\geq 1}\lambda_j(\chi_{j}^2-1)
	\end{equation*}
	and 
	\begin{equation*}
		\pi_n\langle \hat\varphi - \varphi,\mu_1\rangle \xrightarrow{d} N(0,1).
	\end{equation*}
\end{theorem}

Since
\begin{equation*}
	\langle \hat\varphi - \varphi_1,\mu\rangle = \langle \hat\varphi - \varphi_1,\mu_1\rangle  + \langle \hat\varphi - \varphi_1,\mu_0\rangle,
\end{equation*}
the following corollary is a trivial consequence of Theorem~\ref{thm:inner_products_flir}.

\begin{corollary}\label{cor:inner_products_flir}
	Under Assumptions~\ref{as:source4}, \ref{as:data4}, and \ref{as:tuning_inner}, we have the following:
	\begin{enumerate}
		\item If $\pi_n = o(n\alpha_n)$, then
		\begin{equation*}
			\pi_n\langle\hat{\varphi} - \varphi_1,\mu\rangle \xrightarrow{d} N(0,1).
		\end{equation*}
		\item If $n\alpha_n = o(\pi_n)$, then 
		\begin{equation*}
			n\alpha_n\langle\hat{\varphi} - \varphi_1,\mu\rangle \xrightarrow{d} \left(\E\left[\|W\|^2(Y - \langle Z,\varphi_1\rangle)\langle Z,\mu_0\rangle\right] + \sum_{j\geq 1}\lambda_j(\chi_{j}^2-1)\right).
		\end{equation*}
		\item If $\pi_n=n\alpha_n$, then
		\begin{equation*}
			n\alpha_n\langle\hat{\varphi} - \varphi_1,\mu\rangle \xrightarrow{d} N(0,1) + \left(\E\left[\|W\|^2(Y - \langle Z,\varphi_1\rangle)\langle Z,\mu_0\rangle\right] + \sum_{j\geq 1}\lambda_j(\chi_{j}^2-1)\right).
		\end{equation*}
	\end{enumerate}
\end{corollary}
The normalizing sequence in Corollary~\ref{cor:inner_products_flir} can be either $\pi_n$ or $n\alpha_n$, and the asymptotic distribution may feature a Gaussian component, a chi-square component, or both, depending on the mapping properties of the operators $K$ and $\Sigma$. Specifically, if the completeness condition holds, the asymptotic distribution consists solely of a Gaussian component. However, if the completeness condition fails, a chi-square component may also appear.

To illustrate this, we let $W=(W(t))_{t\in[0,1]}$ be a Brownian bridge with a Karhunen--Lo\`{e}ve decomposition $W(t) = \sum_{j=1}^\infty\frac{\xi_j}{j\pi}h_j(t)$, where $\xi_j\sim_{i.i.d.}N(0,1)$ and $h_j(t)=\sqrt{2}\sin(\pi j t)$; see \cite{shorack1986empirical}, p.33. We let $Z(s) = \sum_{j=1}^\infty\frac{\zeta_j}{j\pi}e_j(s)$ be a Gaussian process with $\zeta_j\sim_{i.i.d.}N(0,1)$, $\Cov(\zeta_j,\xi_k)=\rho\mathbf{1}_{j\ne k}$, and $e_j(s)=\sqrt{2}\cos(\pi js)$. In the high-dimensional IV regression, we have
\begin{equation*}
	Y = \int_0^1\varphi(s)Z(s)\dx s + U,\qquad \E[UW(t)] = 0,\qquad \forall t\in[0,1],
\end{equation*}
with $\varphi(s) = \sum_{j=1}^\infty j^{-a}e_j(s),a>1/2$ and $e_j(s)=\sqrt{2}\cos(\pi j s)$. We suppose that the object of interest is the average structural function
\begin{equation*}
	\int_0^1\varphi(s)\mu(s)\dx s = \langle \varphi,\mu\rangle
\end{equation*} 
for $\mu\in L_2[0,1]$ with $\mu(s)=1 + \sum_{j=1}^\infty j^{-b}e_j(s)$ and $b>1/2$.\footnote{Note that $\int_0^1\mu(s)\dx s=1$ and that $\varphi\in L_2[0,1]$.} The covariance operator $K=\E[W\otimes Z]:L_2[0,1]\to L_2[0,1]$ simplifies to $K=\sum_{j=1}^\infty\frac{\rho}{j^2\pi^2}h_j\otimes e_j$ and its adjoint to $K^* = \sum_{j=1}^\infty\frac{\rho}{j^2\pi^2}e_j\otimes h_j$. It can be observed that $1\in\mathcal{N}(K)$, $\varphi_0=0$, and $\mu_0=1$, so the structural function $\varphi$ and average structural function $\langle\varphi,\mu\rangle$ are identified even though $K$ has a nontrivial null space.

To verify Assumption~\ref{as:tuning_inner}, we observe that by Parseval's identity,
\begin{equation*}
	\left\|(K^*K)^{-\beta/2}\varphi_1\right\|^2 = \frac{\pi^{4\beta}}{\rho^{2\beta}}\sum_{j=1}^\infty\frac{1}{j^{2a-4\beta}}<\infty\iff a-2\beta>1/2
\end{equation*}
and
\begin{equation*}
	\left\|(K^*K)^{-\gamma}\mu_1\right\|^2 = \frac{\pi^{8\gamma}}{\rho^{4\gamma}}\sum_{j=1}^\infty\frac{1}{j^{2b-8\gamma}}<\infty \iff b-4\gamma>1/2;
\end{equation*}
cf. Section~\ref{sec:rates}. Under homoskedasticity, we have
$\Sigma = \sigma^2\E[W\otimes W]=\sum_{j=1}^\infty\frac{\sigma^2}{j^2\pi^2}h_j\otimes h_j$, which implies that
\begin{equation*}
	\|\Sigma^{1/2}K(\alpha_n + K^*K)^{-1}\mu_1\|^2 = \frac{\sigma^2}{\pi^2}\sum_{j=1}^\infty \frac{1}{j^{2(b+1)}}\frac{\sigma_j^2}{(\alpha_n + \sigma_j^2)^2}
\end{equation*}
with $\sigma_j=\frac{\rho}{j^2\pi^2}$. Therefore, we obtain
\begin{equation*}
	\pi_n^2\sim \frac{n}{\sum_{j=1}^\infty j^{-2(b+3)}(\alpha_n+j^{-4})^{-2}}.
\end{equation*}

For $\gamma=1/2$, we require that the function $\mu_1$ has sufficient regularity, with $b>5/2$. For larger values of $\beta$, more regularity in $\varphi_1$ is needed, which corresponds to larger values of $a$. In this case, $\pi_n^2\sim n$, and Assumption~\ref{as:tuning_inner} (ii) requires $n\alpha_n^2\to 0$, ensuring that $n\alpha_n=o(\pi_n)$. Consequently, Corollary~\ref{thm:inner_products_flir} shows that
\begin{equation*}
	n\alpha_n\langle\hat\varphi - \varphi,\mu\rangle \xrightarrow{d}  \E\left[U\langle Z,\mu_0\rangle\|W\|^2\right] + \sum_{j=1}^\infty\lambda_j(\chi^2_j-1).
\end{equation*}

\section{Monte Carlo Experiments}\label{sec:mc}
In this section, we examine the validity of our asymptotic theory via Monte Carlo simulations. The data-generating process (DGP) is specified as follows:
\begin{equation*}
	\begin{aligned}
		\varphi(z) & = \sum_{j=1}^\infty(-1)^jj^{-6}e_j(z),\qquad Y = \E[\varphi(Z)|W] + U,\qquad U\sim 0.1N(0,1),\\
		f_{ZW}^J(z,w) & = C_J\sum_{j=1}^{J}j^{-3}e_j(z)e_j(w),\qquad \forall z,w\in[0,1],
	\end{aligned}
\end{equation*}
where we use the basis $e_j(x)=\sqrt{2}\sin(j\pi x),j\geq 1$ and set $C_J$ such that the density integrates to 1; see also \cite{hall2005nonparametric}. The parameter $J$ controls the degree of identification, with $K:L_2[0,1]\to L_2[0,1]$ having an infinite-dimensional null space whenever $J<\infty$. Specifically, $K$ has a $J$-dimensional range and
\begin{equation*}
	\mathcal{N}(K) \subset \mathrm{span}\{e_j:\; j> J\}.
\end{equation*}
Accordingly, we consider three cases: $J\in\{1,2,\infty\}$, where $J=\infty$ corresponds to the fully identified model.

\begin{table}[ht]
	\centering
	\begin{tabular}{lccccc}
		\hline
		\hline
		& \multicolumn{2}{l}{$n=100$} & \multicolumn{2}{l}{$n=500$} \\ \hline
		$J$ & $L_2$ & $L_\infty$ & $L_2$ & $L_\infty$ \\
		\hline
		1 & 0.0378 & 0.3295 & 0.0162 & 0.2156 \\
		2 & 0.0345 & 0.3199 & 0.0130 & 0.2054 \\
		$\infty$ & 0.0311 & 0.3109 & 0.0107 & 0.1974 \\
		\hline
		\hline
	\end{tabular}
	\caption{$L_2$ and $L_\infty$ errors calculated from $5,000$ experiments.}
	\label{tab:1}
\end{table}

\begin{figure}[ht]
	\begin{subfigure}[b]{0.42\textwidth}
		\includegraphics[width=\textwidth]{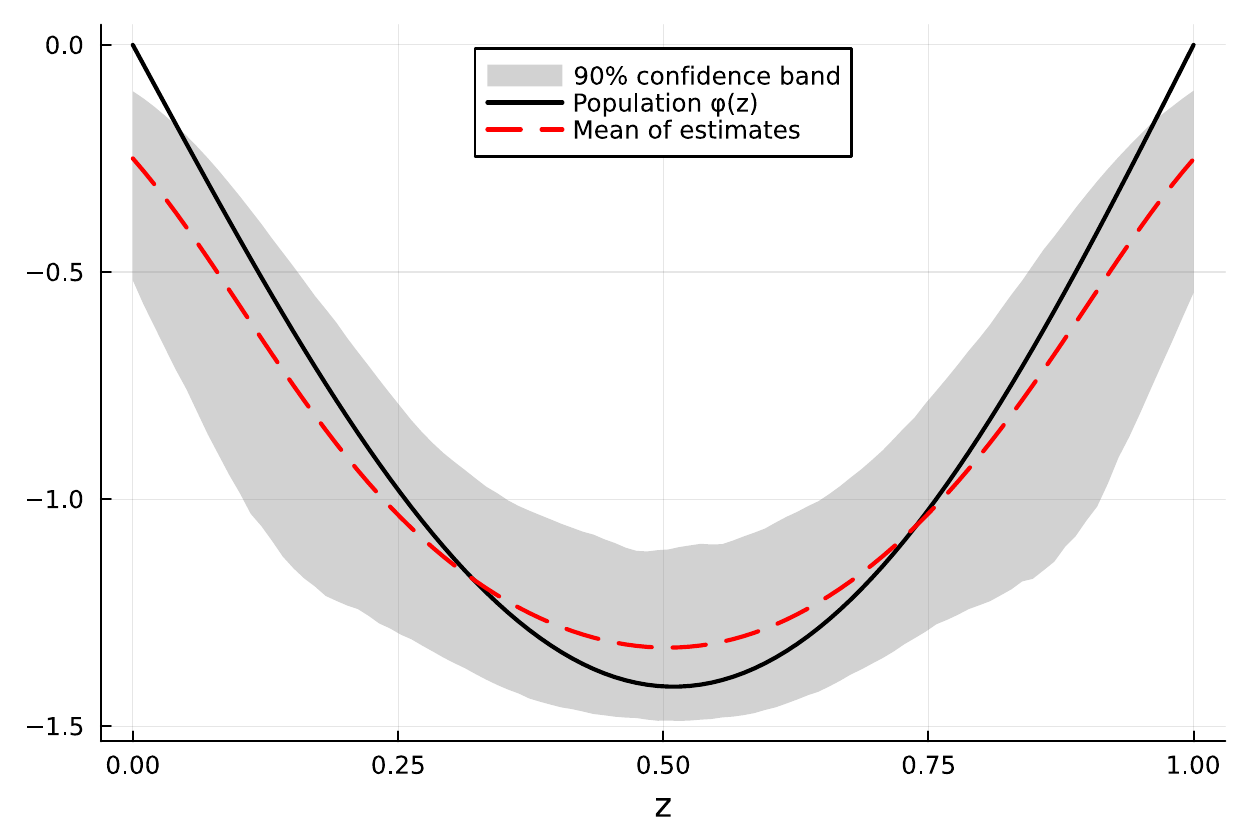}
		\caption{$n = 100$ and $J=1$}
	\end{subfigure}
	\begin{subfigure}[b]{0.42\textwidth}
		\includegraphics[width=\textwidth]{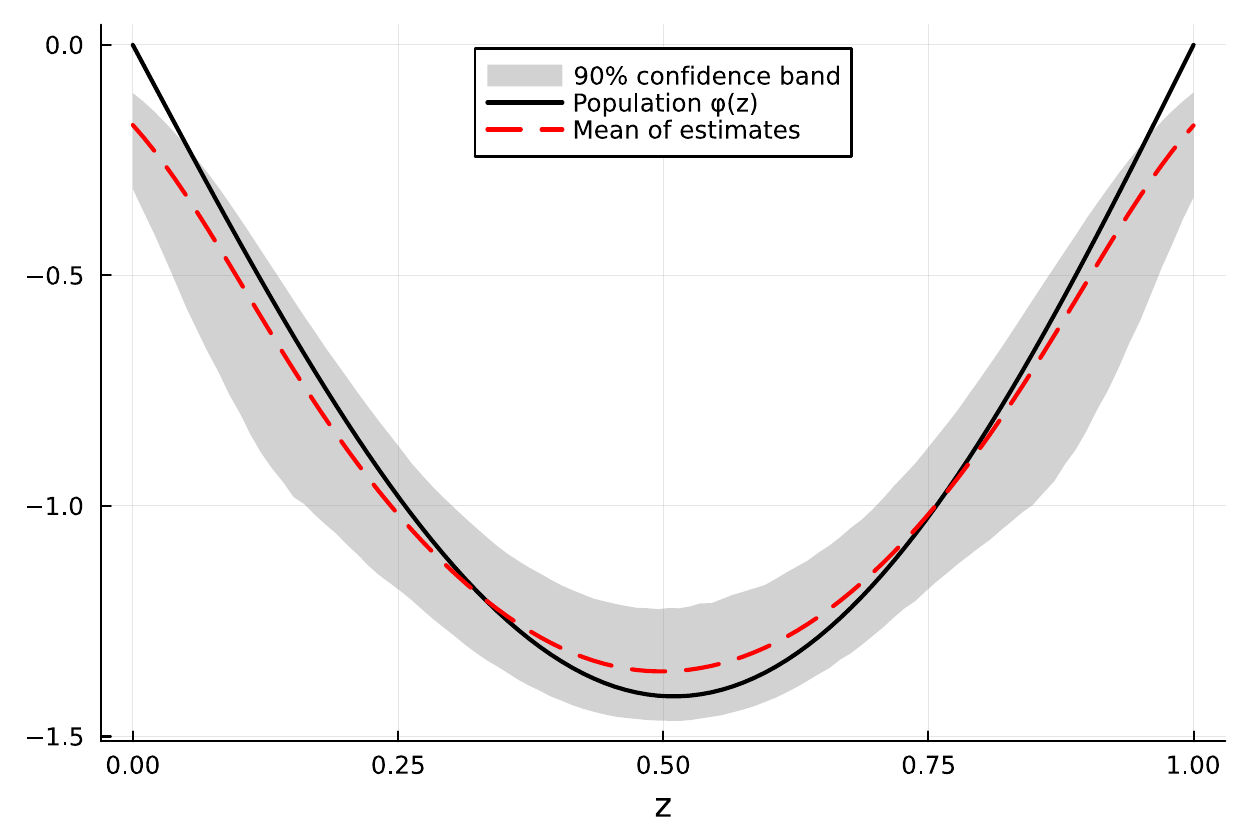}
		\caption{$n = 500$ and $J=1$}
	\end{subfigure}
	\begin{subfigure}[b]{0.42\textwidth}
		\includegraphics[width=\textwidth]{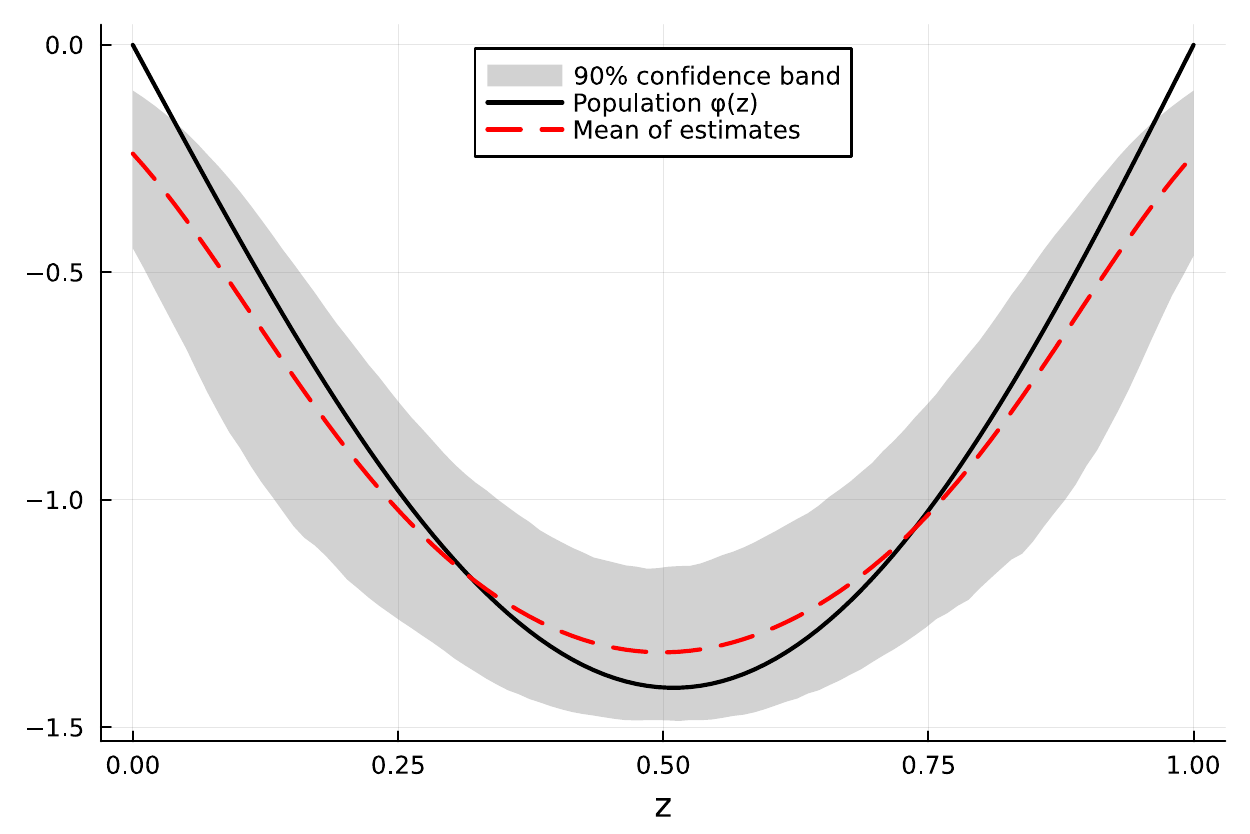}
		\caption{$n = 100$ and $J=\infty$}
	\end{subfigure}
	\qquad\qquad\qquad
	\begin{subfigure}[b]{0.42\textwidth}
		\includegraphics[width=\textwidth]{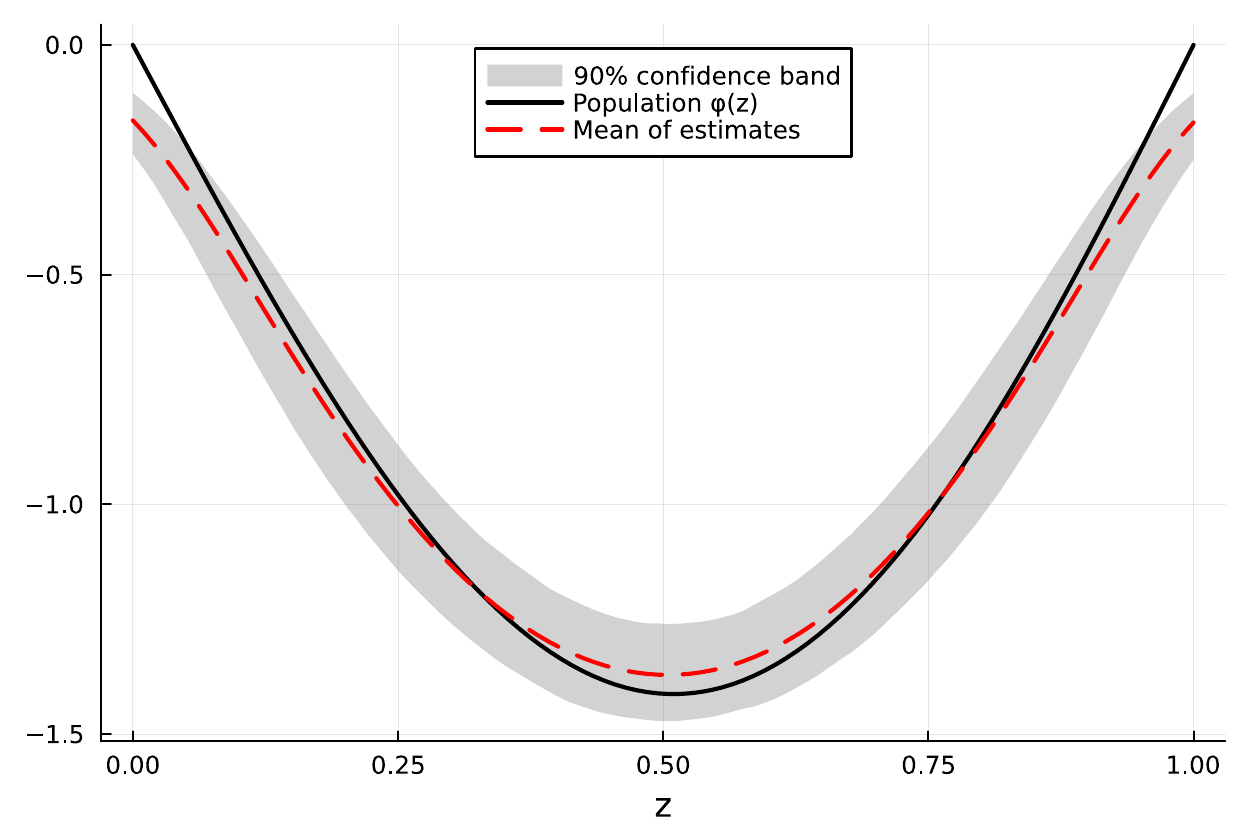}
		\caption{$n = 500$ and $J=\infty$}
	\end{subfigure}
	\caption{Mean estimates with pointwise $90\%$ empirical confidence bands calculated from $5,000$ experiments.}
	\label{fig:1}
\end{figure}

We simulate samples of size $n\in\{100,500\}$ with $5,000$ replications. For simplicity, we focus on the Tikhonov-regularized estimator with a product of Gaussian kernels to estimate the best approximation $\varphi_1$. The bandwidth parameter is selected via Silverman's rule of thumb, and the regularization parameter is chosen via leave-one-out cross-validation; see \cite{babii2020honest} and \cite{centorrino2014data} for details on the numerical implementation. All the integrals are approximated by Riemann sums on a uniform grid in $[0,1]$, and infinite sums are truncated at a large finite value beyond which the numerical results remain unchanged.

Table~\ref{tab:1} shows the empirical $L_2$ and $L_\infty$ errors for the three degrees of identification, as measured by $J\in\{1,2,\infty\}$. The estimation errors decrease as the dimension of the null space shrinks, for both $n=100$ and $n=500$. Larger sample sizes also result in smaller errors, which is consistent with our asymptotic theory. Figure~\ref{fig:1} displays the true structural function $\varphi$, the mean of the estimates $\E\hat\varphi$, and the pointwise 90\% empirical confidence bands based on 5,000 replications. The results align with those in Table~\ref{tab:1}. Note that for $J=1$, we can only recover information related to the first basis vector, illustrating an extreme failure of the completeness condition. Nevertheless, even in this case, a significant amount of information about the structural function can be learned, especially when $n=500$.

\section{Conclusions}\label{sec:conclusion}
This paper develops a theory for nonidentified linear models, using nonparametric IV regression and high-dimensional regressions as illustrative examples. Identification failures arise because of the noninjectivity of covariance or conditional expectation operators. We show that when these operators are noninjective, spectral-regularized estimators converge to the best approximation of the structural parameter in the orthogonal complement of the operator's null space, and we derive new bounds on the risk in both uniform and Hilbert space norms.

This framework offers an appealing projection interpretation for nonparametric IV regression under identification failure, which is analogous to that of ordinary least squares under misspecification. We describe cases where the best approximation coincides with or closely approximates the structural parameter, highlighting the usefulness of these results in partial identification settings. Additionally, incorporating smoothness constraints into the regularization procedure can yield more precise estimates in small samples; see \cite{babii2017unobservables} for an approach using Sobolev-type penalties. Finally, under identification failure, the asymptotic distribution of a continuous linear functional can transition between weighted chi-square and Gaussian components.

{ To conclude, despite a large amount of research over the last two decades, there is only a small amount of empirical work using nonparametric IV regression. One reason is that it involves identifying conditions that are difficult to validate empirically and our paper shows that numerous results may be obtained without these conditions. Economic analysis often hinges on strong functional restrictions that may lead to erroneous counterfactual conclusions and policy recommendations. Nonparametric IV regression is useful for obtaining robust and reliable estimates; see e.g. \cite{compiani2022market}.}

\bibliography{references}

\newpage
\setcounter{page}{1}
\setcounter{section}{0}
\setcounter{equation}{0}
\setcounter{table}{0}
\setcounter{figure}{0}
\renewcommand{\theequation}{A.\arabic{equation}}
\renewcommand\thetable{A.\arabic{table}}
\renewcommand\thefigure{A.\arabic{figure}}
\renewcommand\thesection{A.\arabic{section}}
\renewcommand\thepage{Appendix - \arabic{page}}
\renewcommand\thetheorem{A.\arabic{theorem}}

\begin{center}
	{\LARGE\textbf{APPENDIX}}	
\end{center}
\bigskip

\section{Spectral Regularization}\label{sec:general}
In this section, we extend our results to encompass a range of estimators based on spectral regularization. This framework includes iterated ridge (Tikhonov) regression, functional gradient descent (Landweber-Fridman), and functional PCA (spectral cut-off) methods. Consider the linear operator equation:
\begin{equation*}
K\varphi = r,
\end{equation*}
where $K:\mathcal{E}\to\mathcal{H}$ is a linear operator between Hilbert spaces $\mathcal{E}$ and $\mathcal{H}$. We assume that $K$ is bounded, with $\|K\|^2\leq \Lambda$ for some $\Lambda<\infty$, though it may not be compact. Then $K^*K:\mathcal{E}\to\mathcal{E}$ is a normal operator with the spectral decomposition:
\begin{equation*}
K^*K = \int_{\sigma(K^*K)}\lambda\dx E(\lambda),
\end{equation*}
where $\sigma(K^*K)$ denotes the spectrum of $K^*K$ and $E$ is the resolution of the identity; see \cite{rudin1991functional}, Theorem 12.23. For a bounded Borel function $g:[0,\Lambda]\to \R$, we define:
\begin{equation*}
	g(K^*K) = \int_{\sigma(K^*K)} g(\lambda)\dx E(\lambda).
\end{equation*}
If $K$ is compact, the spectrum of $K^*K$ is countable, and we have
\begin{equation*}
g(K^*K) = \sum_{j=1}^\infty g(\sigma_j^2)E_j,
\end{equation*}
where $E_j=h_j\otimes e_j$ projects on the eigenspace corresponding to the eigenvalue $\sigma_j^2$. 

We aim to recover the best approximation to the structural parameter $\varphi$ when estimates $(\hat K,\hat r)$ of $(K,r)$ are available with $\|\hat K\|^2\leq\Lambda$ almost surely. To do this, we consider an extended version of Assumption~\ref{as:source4}:
\begin{assumption}\label{as:general_source}
	Suppose that $(\varphi,K)$ belongs to
	\begin{equation*}
		\mathcal{F}(\beta,C)=\left\{(\varphi,K):\; \varphi_1=s_\beta(K^*K)\psi,\;\|\psi\|^2\vee\|\varphi_0\|\vee\|K\|\leq C\right\},
	\end{equation*}
	where $s_\beta:[0,\Lambda]\to\R$ is a non-decreasing positive function such that $\lambda\mapsto s_\beta^2(\lambda)/\lambda^{\beta}$ is non-increasing.
\end{assumption}
The following two cases are of particular interest:
\begin{enumerate}
	\item Mildly ill-posed problem: $s_\beta(\lambda) = \lambda^{\beta/2}$.
	\item Severely ill-posed problem: $s_\beta(\lambda) = \log^{-\beta/2}\left(\frac{1}{\lambda}\right)$ with $s_\beta(0)=0$.
\end{enumerate}
In the mildly ill-posed case, the eigenvalues of $K^*K$ can decline exponentially, provided that the Fourier coefficients of $\varphi_1$ decay at the same rate. In contrast, the severely ill-posed case accommodates less regular $\varphi_1$ and $K$.

The spectral regularization scheme is defined by a family of bounded Borel functions $g_\alpha: [0,\infty)\to\R$, where $\alpha>0$ is a regularization parameter such that $\lim_{\alpha\to 0}g_\alpha(\lambda)=\lambda^{-1}$. When $\hat K$ is bounded, the regularized estimator is defined by
\begin{equation*}
	\hat{\varphi} = g_\alpha(\hat K^*\hat K)\hat K^*\hat r.
\end{equation*}
Our theoretical results require the following assumption on the regularization scheme:
\begin{assumption}\label{as:regularization_schemes}
	There exist $c_1,c_2,c_3,\beta_0>0$ such that: (i) $\sup_{\lambda\leq\Lambda }|g_\alpha(\lambda)\lambda^{1/2}| \leq c_1/\sqrt{\alpha}$; (ii) $\sup_{\lambda\leq\Lambda}|(g_\alpha(\lambda)\lambda - 1)\lambda^r| \leq c_2\alpha^r,\forall r\in[0,2\beta_0]$; and (iii) $\sup_{\lambda\leq\Lambda}|g_\alpha(\lambda)| \leq c_3/\alpha$.
\end{assumption}

The following regularization schemes satisfy Assumption~\ref{as:regularization_schemes} in  both the mildly and severely ill-posed cases:\footnote{In the severely ill-posed case, the function $\lambda\mapsto s_\beta^2(\lambda)/\lambda^\beta$ is nonincreasing only on $(0,1/e]$ and $s_\beta$ is not defined at $\lambda=1$. To address these issues, we assume that the norm on $\mathcal{H}$ is scaled such that $\|K\|^2\leq 1/e$.}
\begin{enumerate}
	\item Ridge (Tikhonov):
	\begin{equation*}
	g_\alpha(\lambda) = \frac{1}{\alpha + \lambda}.
	\end{equation*}
	Assumption~\ref{as:regularization_schemes} holds with $c_1 = 1/2$, $c_2=c_3=1$, and $\beta_0=2$.
	\item PCA (spectral cut-off):
	\begin{equation*}
	g_\alpha(\lambda) = \lambda^{-1}\one\{\lambda\geq\alpha\}.
	\end{equation*}
	Assumption~\ref{as:regularization_schemes} holds with $c_1=c_2=c_3=1$ and every $\beta_0>0$.
	\item Iterated Tikhonov/ridge regularization:
	\begin{equation*}
		g_{\alpha,m}(\lambda) = \sum_{j=0}^{m-1}\frac{\alpha^j}{(\alpha + \lambda)^{j+1}} = \frac{1}{\lambda}\left(1 - \left(\frac{\alpha}{\lambda+\alpha}\right)^{m}\right),\qquad m=2,3,\dots
	\end{equation*}
	Assumption~\ref{as:regularization_schemes} holds with $c_1 =m^{1/2}$, $c_2=1$, and $c_3=\beta_0=m$.
	\item Gradient descent (Landweber-Fridman):
	\begin{equation}\label{eq:gradient_regularization}
		g_{\alpha,c}(\lambda) = \sum_{j=0}^{1/\alpha - 1}(1-c\lambda)^j = \frac{1}{\lambda}\left(1 - (1-c\lambda)^{1/\alpha}\right),
	\end{equation}
	where $\alpha = 1/m$ for some $m\in\Nn$ and $c\in(0,1/\Lambda)$. Assumption~\ref{as:regularization_schemes} is satisfied with $c_1^2=c$, $c_2=c\vee 1$, $c_3=\left(\frac{\beta}{ce}\right)^{\beta/2}\vee 1$, and every $\beta_0>0$.
\end{enumerate}
The constant $\beta_0$ is known as the qualification of the regularization scheme. It is well-established that simple Tikhonov regularization experiences a saturation effect, limiting bias convergence to $O(\alpha^2)$. This issue can be resolved using iterated Tikhonov regularization, similar to higher-order kernels in nonparametric kernel estimation. 

To provide a spectral interpretation of gradient descent in point 4, consider the minimization problem:
\begin{equation*}
	\min_{\phi}\underbrace{\frac{1}{2}\|K\phi - r\|^2}_{=:Q(\phi)},
\end{equation*}
where the gradient is $\nabla Q(\phi):=K^*(K\phi - r)$. The gradient descent with a constant step size $c\in(0,1/\Lambda)$, leads to the update rule:
\begin{equation*}
	\begin{aligned}
		\phi_m & = \phi_{m-1} - c\nabla Q(\phi_{m-1}) \\
		& = [I-cK^*K]\phi_{m-1} + cK^*r.
	\end{aligned}	
\end{equation*}
Iterating this update rule recursively, we obtain
\begin{equation*}
	\phi_m = [I-cK^*K]^m\phi_0 + \sum_{j=0}^{m-1}[I-cK^*K]^jK^*r.
\end{equation*}
Assuming that the initial value is $\phi_0=0$ and replacing $(K,r)$ by $(\hat K,\hat r)$, we obtain 
\begin{equation*}
	\hat\varphi_{\alpha} = g_{\alpha,c}(\hat K^*\hat K)\hat K^*\hat r
\end{equation*}
with $g_{\alpha,c}$ and $\alpha=1/m$ as in equation~(\ref{eq:gradient_regularization}).

The following theorem establishes the convergence rate of $\hat\varphi$ to $\varphi_1$ for general regularization schemes:
\begin{theorem}\label{thm:L2_rates_general}
	Suppose Assumptions~\ref{as:general_source}, \ref{as:regularization_schemes}, and \ref{as:rates4} hold with $\beta\leq\beta_0$.\footnote{Note that we can always take $\beta=\beta_0$ if $\beta>\beta_0$.} Then, in the mildly ill-posed case,
	\begin{equation*}
		\left\|\hat \varphi - \varphi_1\right\|^2 = O_P\left(\frac{\delta_{n}}{\alpha_n} + \rho_{1n}^{\beta\wedge1}\left(1 + \one_{\beta=1}\log^2\rho_{1n}^{-1} \right) + \alpha_n^\beta\right),
	\end{equation*}
	and in the severely ill-posed case,
	\begin{equation*}
		\left\|\hat \varphi - \varphi_1\right\|^2 = O_P\left(\frac{\delta_{n}}{\alpha_n} + \log^{-\beta}\rho_{1n}^{-1} + \log^{-\beta}\alpha_n^{-1}\right).
	\end{equation*}
\end{theorem}
Before presenting the proof, note that:
\begin{enumerate}
	\item For the mildly ill-posed case, the optimal choice for the regularization parameter is $\alpha_n\sim \delta_n^\frac{1}{\beta+1}$, assuming that $\rho_{1n}^{\beta\wedge 1}\log^2n\lesssim \delta_n/\alpha_n$. This selection results in the convergence rate:
	\begin{equation*}
		\|\hat\varphi - \varphi_1\|^2 = O_P\left(\delta_n^\frac{\beta}{\beta+1}\right).
	\end{equation*}
	\item For the severely ill-posed case, one can select $\alpha_n\sim \delta_n^{1/2}$. This choice yields the convergence rate:
	\begin{equation*}
		\|\hat\varphi - \varphi_1\|^2 = O_P\left(\frac{1}{\log^\beta n}\right),
	\end{equation*}
	assuming $\delta_n,\rho_n\sim n^{-c}$ for some $c>0$.
\end{enumerate}

\begin{proof}[Proof of Theorem~\ref{thm:L2_rates_general}]
	Decompose
	\begin{equation*}
	\hat{\varphi} - \varphi_1 = I_n + II_n + III_n,
	\end{equation*}
	where
	\begin{equation*}
	\begin{aligned}
	I_n & = g_{\alpha}(\hat K^*\hat K)\hat K^*(\hat r - \hat K\varphi), \\
	II_n & = \left[g_{\alpha}(\hat K^*\hat K)\hat K^*\hat K - I\right]s_\beta(\hat K^*\hat K)\psi, \\
	III_n & = \left[g_{\alpha}(\hat K^*\hat K)\hat K^*\hat K - I\right]\left\{s_\beta(K^*K) - s_\beta(\hat K^*\hat K)\right\}\psi.
	\end{aligned}
	\end{equation*}
	To verify this decomposition, observe that $\varphi = \varphi_1+\varphi_0$ and, under Assumption~\ref{as:general_source}, $\varphi_1 = s_\beta(K^*K)\psi$. Using the isometry of functional calculus, we bound each term as follows:
	\begin{equation*}
	\begin{aligned}
	\|I_n\|^2 & \leq \left\|g_{\alpha}(\hat K^*\hat K)\hat K^*\right\|^2\left\|\hat r - \hat K\varphi_1\right\|^2 \\
	& \leq \sup_{\lambda\leq\Lambda}\left|g_{\alpha}(\lambda)\lambda^{1/2}\right|^2\left\|\hat r - \hat K\varphi_1\right\|^2.
	\end{aligned}
	\end{equation*}
	Under Assumptions~\ref{as:regularization_schemes} (i) and \ref{as:rates4}, this shows that $\E\|I_n\|^2 = O\left(\delta_{n}/\alpha_n\right)$.
	
	Note also that under Assumption~\ref{as:general_source}, $s_\beta$ is nondecreasing, whence under Assumption~\ref{as:regularization_schemes} (ii) with $r=0$, we obtain
	\begin{equation*}
	\left|(g_{\alpha}(\lambda)\lambda - 1)s_\beta(\lambda)\right|\leq c_2|s_\beta(\lambda)| \leq c_2s_\beta(\alpha_n),\qquad \forall \lambda\in[0,\alpha_n].
	\end{equation*}
	Similarly, since $\lambda\mapsto s_\beta(\lambda)/\lambda^{\beta/2}$ is nonincreasing, under Assumption~\ref{as:regularization_schemes} (ii) with $r=\beta/2$
	\begin{equation*}
	\left|(g_{\alpha}(\lambda)\lambda - 1)s_\beta(\lambda)\right| \leq \left|(g_{\alpha}(\lambda)\lambda - 1)\lambda^{\beta/2}\right|\left|\frac{s_\beta(\lambda)}{\lambda^{\beta/2}}\right| \leq c_2\alpha_n^{\beta/2}s_\beta(\alpha_n)/\alpha_n^{\beta/2} ,\qquad \forall \lambda\geq\alpha_n.
	\end{equation*}
	Thus,
	\begin{equation*}
		\begin{aligned}
			\|II_n\|^2 & \leq \left\|\left[g_{\alpha}(\hat K^*\hat K)\hat K^*\hat K - I\right]s_\beta(\hat K^*\hat K)\right\|^2\left\|\psi\right\|^2 \\
			& \leq \sup_{\lambda\leq\Lambda}\left|(g_{\alpha}(\lambda)\lambda - 1)s_\beta(\lambda)\right|^2C. \\
		\end{aligned}
	\end{equation*}
	
	Finally, under Assumption~\ref{as:regularization_schemes} (ii) with $r=0$
	\begin{equation*}
	\begin{aligned}
	\|III_n\|^2 & \leq \left\|g_{\alpha}(\hat K^*\hat K)\hat K^*\hat K - I\right\|^2 \left\|s_\beta(\hat K^*\hat K) - s_\beta(K^*K)\right\|^2\left\|\psi\right\|^2 \\
	& \leq C\sup_{\lambda\leq\Lambda}\left|(g_{\alpha}(\lambda)\lambda - 1)\right|^2\left\|s_\beta(\hat K^*\hat K) - s_\beta(K^*K)\right\|^2 \\
	& \leq Cc_2^2\left\|s_\beta(\hat K^*\hat K) - s_\beta(K^*K)\right\|^2.
	\end{aligned}
	\end{equation*}
	The result follows by applying Lemma~\ref{prop:delta_method} which appears at the end of this section.
\end{proof}

The following theorem establishes the uniform convergence rates for a broad class of regularized estimators:
\begin{theorem}\label{thm:Linf_rates_general}	
	Suppose that Assumptions~\ref{as:general_source}, \ref{as:regularization_schemes}, \ref{as:rates4}, and \ref{as:rates_2-infty} are satisfied with $\varphi_1 = s_\beta(K^*K)K^*\psi$ and $\beta_0\geq \beta$. Then, in the mildly ill-posed case, 
	\begin{equation*}
		\left\|\hat \varphi - \varphi_1\right\|_\infty = O_P\left(\frac{\delta_{n}^{1/2}}{\alpha_n} + \frac{1}{\alpha_n^{1/2}}\left(\rho_{2n}^{1/2} + \rho_{1n}^{\frac{\beta\wedge1}{2}}\left(1 + \one_{\beta=1}\log\rho_{1n}^{-1} \right) \right) + \alpha_n^{\beta/2}\right),
	\end{equation*}
	and in the severely ill-posed case,
	\begin{equation*}
		\left\|\hat \varphi - \varphi_1\right\|_\infty = O_P\left(\frac{\delta_n^{1/2}}{\alpha_n} + \frac{1}{\alpha_n^{1/2}}\left(\rho_{2n}^{1/2} + \log^{-\beta/2}\rho_{1n}^{-1} \right) + \log^{-\beta/2}\alpha_n^{-1}\right).
	\end{equation*}
\end{theorem}

\begin{proof}
	We start with the decomposition 
	\begin{equation*}
		\hat{\varphi} - \varphi_1 = I_n + II_n + III_n,
	\end{equation*}
	where
	\begin{equation*}
	\begin{aligned}
		I_n & = g_{\alpha}(\hat K^*\hat K)\hat K^*(\hat r - \hat K\varphi_1), \\
		II_n & = \left[g_{\alpha}(\hat K^*\hat K)\hat K^*\hat K - I\right]s_\beta(\hat K^*\hat K)\hat K^*\psi, \\
		III_n & = \left[g_{\alpha}(\hat K^*\hat K)\hat K^*\hat K - I\right]\left\{s_\beta(K^*K)K^* - s_\beta(\hat K^*\hat K)\hat K^*\right\}\psi.
	\end{aligned}
	\end{equation*}
	We proceed by bounding each term individually. For the first term, we have
	\begin{equation*}
	\begin{aligned}
	\|I_n\|_\infty & = \left\|\hat K^*g_{\alpha}(\hat K\hat K^*)(\hat r - \hat K\varphi_1)\right\| \\
	& \leq \|\hat K^*\|_{2,\infty}\sup_{\lambda\leq\Lambda}|g_\alpha(\lambda)|\left\|\hat r - \hat K\varphi_1\right\| \\
	& = O_P\left(\frac{\delta_n^{1/2}}{\alpha_n}\right),
	\end{aligned}
	\end{equation*}
	where the last line follows under Assumptions~\ref{as:rates_2-infty} and \ref{as:regularization_schemes} (iii).
	
	For the second term, note that
	\begin{equation*}
	\begin{aligned}
	\|II_n\|_\infty & = \left\|\hat K^*\left[g_{\alpha}(\hat K\hat K^*)\hat K\hat K^* - I\right]s_\beta(\hat K\hat K^*)\psi\right\| \\
	& \leq \|\hat K^*\|_{2,\infty}\sup_{\lambda\leq\Lambda}|(g_\alpha(\lambda)\lambda - 1)s_\beta(\lambda)|\|\psi\| \\
	& = O_P(s_\beta(\alpha_n)),
	\end{aligned}
	\end{equation*}
	where the last line follows by the same argument as in the proof of Theorem~\ref{thm:L2_rates_general} under Assumptions~\ref{as:rates_2-infty} and \ref{as:regularization_schemes} (ii). 
	
	Next, for the third term,
	\begin{equation*}
	\|III_n\|_\infty \leq \left\|g_{\alpha}(\hat K^*\hat K)\hat K^*\hat K - I\right\|_\infty\left\|s_\beta(K^*K)K^* - s_\beta(\hat K^*\hat K)\hat K^*\right\|_{2,\infty} \|\psi\|.
	\end{equation*}
	Under Assumptions~\ref{as:regularization_schemes} (i) and \ref{as:rates_2-infty}
	\begin{equation*}
	\begin{aligned}
	\left\|g_{\alpha}(\hat K^*\hat K)\hat K^*\hat K - I\right\|_\infty & \leq \|\hat K^*\|_{2,\infty}\left\|g_{\alpha}(\hat K\hat K^*)\hat K\right\| + 1 \\
	& \leq\|\hat K^*\|_{2,\infty}\sup_{\lambda\leq \Lambda}\left|g_{\alpha}(\lambda)\lambda^{1/2}\right| + 1 \\
	& = O_P\left(\frac{1}{\alpha_n^{1/2}}\right).
	\end{aligned}
	\end{equation*}
	Finally, under Assumption~\ref{as:rates_2-infty}
	\begin{equation*}
	\begin{aligned}
		\left\|s_\beta(K^*K)K^* - s_\beta(\hat K^*\hat K)\hat K^*\right\|_{2,\infty} & = \left\|K^*s_\beta(KK^*) - \hat K^*s_\beta(\hat K\hat K^*)\right\|_{2,\infty} \\
		& \lesssim \left\|\hat K^* - K^*\right\|_{2,\infty} + \left\|s_\beta(\hat K^*\hat K) - s_\beta(K^*K)\right\| \\
	& \lesssim \rho_{2n}^{1/2} + \left\|s_\beta(\hat K^*\hat K) - s_\beta(K^*K)\right\|.
	\end{aligned}
	\end{equation*}
	The result now follows by applying Lemma~\ref{prop:delta_method} similar to the proof of Theorem~\ref{thm:L2_rates_general}.
\end{proof}

\begin{lemma}\label{prop:delta_method}
	Assume that Assumption~\ref{as:rates4} (iii) holds with $\rho_{1n}\sim n^{-c}$ for some $c>0$. Then
	\begin{equation*}
		\left\|(\hat K^*\hat K)^{\beta/2} - (K^*K)^{\beta/2}\right\| = O_P\left( \left(1 + \one_{\beta=1}\log \rho_{1n}^{-1} \right)\rho_{1n}^\frac{\beta\wedge 1}{2}\right)
	\end{equation*}
	and
	\begin{equation*}
	\left\|\log^{-\beta/2}(\hat K^*\hat K)^{-1} - \log^{-\beta/2}(K^*K)^{-1}\right\| = O_P\left(\log^{-\beta/2}\rho_{1n}^{-1}\right).
	\end{equation*}
\end{lemma}
\begin{proof}
	Using Lemma 3.2 from \cite{egger2005accelerated}, we have:
	\begin{equation*}
		\left\|(\hat K^*\hat K)^{\beta/2} - (K^*K)^{\beta/2}\right\| \lesssim \begin{cases}
			\|\hat K - K\|^\beta, & \beta<1 \\
			\|\hat K - K\|\left\{1 + \|\hat K\| + \|K\| + \log\|\hat K - K\|^{-1}\right\}, & \beta = 1 \\
			\|\hat K - K\|(\|\hat K\| + \|K\|)^{\beta/2}, & \beta>1.
		\end{cases}
	\end{equation*}
	The first statement follows because $\|\hat K - K\|=O_P(\rho_{1n}^{1/2})$, with $\rho_{1n}\to0$ as $n\to\infty$ under Assumption~\ref{as:rates4} (iii), and since $x\mapsto x\log(1/x)$ is strictly increasing near zero.
	
	For the second statement, by Theorem 4 in \cite{mathe2002moduli}, 
	\begin{equation*}
		\begin{aligned}
			\left\|\log^{-\beta/2}(\hat K^*\hat K)^{-1} - \log^{-\beta/2}(K^*K)^{-1}\right\| & = O_P\left(\log^{-\beta/2}\|\hat K^*\hat K - K^*K\|^{-1}\right) \\
			& = O_P\left(\log^{-\beta/2}\rho_{1n}^{-1}\right),
		\end{aligned}
	\end{equation*}
	where the last line uses Assumption~\ref{as:rates4} (iii) and the fact that $x\mapsto \log^{-\beta/2}(1/x)$ is strictly increasing near zero.
\end{proof}

\section{Proofs of Main Results}\label{app:proofs}
\begin{proof}[Proof of Theorem~\ref{thm:L2_rate}]
	Decompose
	\begin{equation*}
	\begin{aligned}
	\hat{\varphi} - \varphi_1 & = (\alpha_n I + \hat K^*\hat K)^{-1}\hat K^*(\hat r - \hat K\varphi_1) \\
	& + (\alpha_n I + \hat K^*\hat K)^{-1}\hat K^*\hat K\varphi_1 - (\alpha_n I + K^*K)^{-1}K^*K\varphi_1 \\
	& + \left((\alpha_n I + K^*K)^{-1}K^*K - I\right)\varphi_1 \\
	& \triangleq I_n + II_n + III_n.
	\end{aligned}
	\end{equation*}
	The term $III_n$ represents the regularization bias, which can be controlled under Assumption~\ref{as:source4} as follows:
	\begin{equation*}
	\begin{aligned}
		\|III_n\|^2 & = \left\|\alpha_n(\alpha_n I + K^*K)^{-1}\varphi_1\right\|^2 \\
		& \leq C\left\|\alpha_n(\alpha_n I + K^*K)^{-1}(K^*K)^{\beta/2}\right\|^2 \\
		& \leq C \sup_{\lambda\in[0,\|K\|^2]}\left|\frac{\alpha_n\lambda^{\beta/2}}{\alpha_n+\lambda}\right|^2 \\
		& \leq C^{(2\beta - 3)\vee 1}\alpha_n^{\beta};
	\end{aligned}
	\end{equation*}
	see \cite{babii2021high} for more details. The first term $I_n$ is controlled under Assumption~\ref{as:rates4} (i)
	\begin{equation*}
	\begin{aligned}
	\E\|I_n\|^2 & \leq \E\left\|(\alpha_n I + \hat K^*\hat K)^{-1}\hat K^*\right\|^2\left\|\hat r - \hat K\varphi_1\right\|^2 \\
	& \leq \sup_{\lambda\geq 0}\left|\frac{\lambda^{1/2}}{\alpha_n + \lambda}\right|^2\E\left\|\hat r - \hat K\varphi_1\right\|^2 \\
	& \leq \frac{1}{4\alpha_n}\E\left\|\hat r - \hat K\varphi_1\right\|^2 \\
	& \leq \frac{C_1\delta_{n}}{4\alpha_n}.
	\end{aligned}
	\end{equation*}
	Next, decompose $II_n$ further as follows:
	\begin{equation*}
	\begin{aligned}
	II_n & = -\left[ \alpha_n (\alpha_n I + \hat K^*\hat K)^{-1} - \alpha_n(\alpha_n I + K^*K)^{-1}\right]\varphi_1 \\
	& = -(\alpha_n I + \hat K^*\hat K)^{-1}\alpha_n\left[K^*K - \hat K^*\hat K\right](\alpha_n I + K^*K)^{-1}\varphi_1 \\
	& = (\alpha_n I + \hat K^*\hat K)^{-1}\hat K^*\left[\hat K - K\right]\alpha_n(\alpha_n I + K^*K)^{-1}\varphi_1 \\
	& \qquad + (\alpha_n I + \hat K^*\hat K)^{-1}\left[\hat K^* - K^*\right]\alpha_n K(\alpha_n I + K^*K)^{-1}\varphi_1 \\
	& = II^a_n + II^b_n. \\
	\end{aligned}
	\end{equation*}
	Now, using previous results and Assumption~\ref{as:rates4} (iii), we have
	\begin{equation*}
	\begin{aligned}
	\E\|II^a_n\|^2 & = \E\left\|(\alpha_n I + \hat K^*\hat K)^{-1}\hat K^*\left[\hat K - K\right]\alpha_n(\alpha_n I + K^*K)^{-1}\varphi_1\right\|^2 \\
	& \leq \E\left\|(\alpha_n I + \hat K^*\hat K)^{-1}\hat K^*\right\|^2\left\|\hat K - K\right\|^2\left\|\alpha_n(\alpha_n I + K^*K)^{-1}\varphi_1\right\|^2 \\
	& \leq \sup_{\lambda\geq 0}\left|\frac{\lambda^{1/2}}{\alpha_n + \lambda}\right|^2\E\left\|\hat K - K\right\|^2C^{(2\beta - 3)\vee 1}\alpha_n^{\beta\wedge 2} \\ 
	& \leq \frac{C_2\rho_{1n}}{4\alpha_n}C^{(2\beta - 3)\vee 1}\alpha_n^{\beta\wedge 2}.
	\end{aligned}
	\end{equation*}
	For $II_n^b$, we obtain
	\begin{equation*}
	\begin{aligned}
	\E\|II^b_n\|^2 & = \E\left\|(\alpha_n I + \hat K^*\hat K)^{-1}\left[\hat K^* - K^*\right]\alpha_n K(\alpha_n I + K^*K)^{-1}\varphi_1\right\|^2 \\
	& \leq \E\left\|(\alpha_n I + \hat K^*\hat K)^{-1}\right\|^2\left\|\hat K^* - K^*\right\|^2C\left\|\alpha_n K(\alpha_n I + K^*K)^{-1}(K^*K)^{\beta/2}\right\|^2 \\
	& \leq \sup_{\lambda\geq 0}\left|\frac{1}{\alpha_n + \lambda}\right|^2\E\left\|\hat K^* - K^*\right\|^2C\sup_{\lambda\in[0,C^2]}\left|\frac{\alpha_n\lambda^{(\beta+1)/2}}{\alpha_n + \lambda}\right|^2\\ 
	& \leq \frac{C_{2}\rho_{1n}}{\alpha_n^2}C^{(2\beta-1)\vee 1}\alpha_n^{(\beta+1)\wedge 2},
	\end{aligned}
	\end{equation*}
	where the third {line} uses $\|\hat K^* - K^*\|=\|\hat K - K\|$. Combining all estimates, we obtain the final result.
\end{proof}

\begin{proof}[Proof of the Theorem~\ref{thm:Linf_rate}]
	Consider the same decomposition as in the proof of Theorem~\ref{thm:L2_rate}. Given that $\varphi_1=(K^*K)^{\beta/2} K^*\psi$, the bias term is handled similarly to the identified case (see \cite{babii2020honest}, Proposition 3.1):
	\begin{equation*}
	\begin{aligned}
	\|III_n\|_\infty & = \left\|\alpha_nK^*(\alpha_nI + KK^*)^{-1}(KK^*)^{\frac{\beta}{2}}\psi\right\|_\infty \\
	& \leq \left\|K^*\right\|_{2,\infty}\left\|\alpha_n(\alpha_n I + KK^*)^{-1}(KK^*)^{\frac{\beta}{2}}\right\|\left\|\psi\right\| \\
	& =O(\alpha_n^{\beta/2}).
	\end{aligned}
	\end{equation*}
	Next, using the Cauchy-Schwartz inequality, Assumption~\ref{as:rates4} (iii), and Assumption~\ref{as:rates_2-infty}, we can bound the first term as follows:
	\begin{equation*}
	\begin{aligned}
	\E\|I_n\|_\infty & = \E\left\|\hat K^*(\alpha_n I + \hat K\hat K^*)^{-1}(\hat r - \hat K\varphi_1)\right\|_\infty \\
	& \leq \E\|\hat K^*\|_{2,\infty}\left\|(\alpha_n I + \hat K\hat K^*)^{-1}\right\|\left\|(\hat r - \hat K\varphi_1)\right\| \\
	& \leq \frac{1}{\alpha_n}\left(\|K^*\|_{2,\infty}\E\left\|(\hat r - \hat K\varphi_1)\right\| + \E\left\|\hat K^* - K^*\right\|_{2,\infty}\left\|(\hat r - \hat K\varphi_1)\right\| \right)\\
	& \leq C_1^{1/2}\left(C_3 + C_3^{1/2}\rho_{2n}^{1/2}\right)\frac{\delta_{n}^{1/2}}{\alpha_n}.
	\end{aligned}
	\end{equation*}
	
	The second term can be decomposed into two parts, $II_n^a$ and $II_n^b$, similarly to the proof of Theorem~\ref{thm:L2_rate}, and we bound each part separately. First, for $II_n^a$:
	\begin{equation*}
	\begin{aligned}
	\E\|II_n^a\|_\infty & = \left\|\hat K^*(\alpha_n I + \hat K\hat K^*)^{-1}\left[\hat K - K\right]\alpha_n(\alpha_n I + K^*K)^{-1}\varphi_1\right\|_\infty \\
	& \leq \E\left\|\hat K^*\right\|_{2,\infty}\left\|(\alpha_n I + \hat K\hat K^*)^{-1}\right\|\left\|\hat K - K\right\|\left\|\alpha_n(\alpha_n I + K^*K)^{-1}\varphi_1\right\| \\
	& \leq \frac{1}{\alpha_n}\left(C_3 + \left(\E\left\|\hat K^* - K^*\right\|_{2,\infty}^2\right)^{1/2}\right)\left(\E\left\|\hat K - K\right\|^2\right)^{1/2} C^{(\beta-1.5)\vee 0.5}\alpha_n^{\beta/2} \\
	& \leq \left(C_3 + C_3^{1/2}\rho_{2n}^{1/2}\right)\frac{C_2^{1/2}\rho_{1n}^{1/2}}{\alpha_n}C^{(\beta-1.5)\vee 0.5}\alpha_n^{\beta/2}.
	\end{aligned}
	\end{equation*}	
	For the second part, $II_n^b$, applying Assumption~\ref{as:rates_2-infty}, and using the inequality in \cite{babii2020honest}, Lemma A.4.1 (see also \cite{nair2009linear}, Problem 5.8), we get:
	\begin{equation*}
	\begin{aligned}
		\E\|II_n^b\|_\infty & = \E\left\|(\alpha_n I + \hat K^*\hat K)^{-1}\left[\hat K^* - K^*\right]\alpha_n K(\alpha_n I + K^*K)^{-1}\varphi_1\right\|_\infty \\
		& = \E\left\|(\alpha_n I + \hat K^*\hat K)^{-1}\right\|_\infty\left\|\hat K^* - K^*\right\|_{2,\infty}\left\|\alpha_n K(\alpha_n I + K^*K)^{-1}\varphi_1\right\| \\
		& = \frac{1}{2\alpha_n^{3/2}}\E\left(\|\hat K^*\|_{2,\infty}  + 2\alpha_n^{1/2}\right)\left\|\hat K^* - K^*\right\|_{2,\infty}C^{(\beta-1/2)\vee 1/2}\alpha_n^{\frac{\beta+1}{2}\wedge 1} \\
		& \leq \frac{1}{2\alpha_n^{3/2}}\left(C_3 + C_3^{1/2}\rho_{2n}^{1/2} + 2\alpha_n^{1/2}\right)C_3^{1/2}\rho_{2n}^{1/2}C^{(\beta-1/2)\vee 1/2}\alpha_n^{\frac{\beta+1}{2}\wedge 1}.
	\end{aligned}
	\end{equation*}
	Combining all these estimates yields the final result.
\end{proof}

The following proposition provides low-level conditions for Assumptions~\ref{as:rates4} and \ref{as:rates_2-infty} in the context of nonparametric IV regression estimated with kernel smoothing. Let $C^s_M$ denote the H\"{o}lder class.
\begin{proposition}\label{prop:npiv_components}
	Suppose that (i) $(Y_i,Z_i,W_i)_{i=1}^n$ are i.i.d. and $\E|Y_1|^2\leq C\infty$; (ii) $f_{ZW}\in C^s_M$; (iii) kernel functions $K_z:\R^p\to\R$ and $K_w:\R^q\to\R$ are such that for $l\in\{w,z\}$, $K_l\in L_1\cap L_2$, $\int K_l(u)\dx u = 1$, $\int\|u\|^sK_l(u)\dx u<\infty$, and $\int u^kK_l(u)\dx u=0$ for all multindices $|k|=1,\dots,\lfloor s\rfloor$. Then
	\begin{equation*}
	\E\left\|\hat r - \hat K\varphi_1\right\|^2 = O\left(\frac{1}{nh^q_n} + h^{2s}_n\right)\qquad \mathrm{and}\qquad \E\left\|\hat K - K\right\|^2 = O\left(\frac{1}{nh^{p+q}_n} + h^{2s}_n\right),
	\end{equation*}
	where the constants do not depend on $(K,\varphi)$.
\end{proposition}
\begin{proof}
	For the first claim, note that
	\begin{equation*}
	\E\left\|\hat r - \hat K\varphi_1\right\|^2 \leq 2\E\left\|\hat r - r\right\|^2 + 2\E\left\|(\hat K - K)\varphi_1\right\|^2.
	\end{equation*}
	Decompose
	\begin{equation*}
	\E\left\|\hat r - r\right\|^2 = \E\left\|\hat r - \E\hat r\right\|^2 + \left\|\E\hat r - r\right\|^2.
	\end{equation*}
	Under the i.i.d. assumption,
	\begin{equation*}
	\begin{aligned}
	\E\left\|\hat r - \E\hat r\right\|^2 & = \E\left\|\frac{1}{nh^q_n}\sum_{i=1}^nY_iK_w\left(h^{-1}_n(W_i-w)\right) - \E\left[Y_ih^{-q}_nK_w\left(h^{-1}_n(W_i - w)\right)\right]\right\|^2 \\
	& = \frac{1}{n}\E\left\|Y_ih^{-q}_nK_w\left(h^{-1}_n(W_i-w)\right) - \E\left[Y_ih^{-q}_nK_w\left(h^{-1}_n(W_i-w)\right)\right]\right\|^2 \\
	& \leq \frac{1}{nh^q_n} \E |Y_1|^2\|K_w\|^2  \\
	& = O\left(\frac{1}{nh^{q}_n}\right).
	\end{aligned}
	\end{equation*}
	Applying the Cauchy-Schwartz inequality, we obtain
	\begin{equation*}
	\begin{aligned}
		\E\hat r - r & = \E\left[\varphi(Z_i)h^{-q}_nK_w\left(h^{-1}_n(W_i-w)\right)\right] - \int\varphi(z)f_{ZW}(z,w)\dx z \\
		& = \int\varphi(z)\left\{ [f_{ZW}\ast K_w	](z,w) - f_{ZW}(z,w)\right\}\dx z \\
		& \leq \|\varphi\|\left\|f_{ZW}\ast K_w - f_{ZW}\right\|,
	\end{aligned}
	\end{equation*}
	where $[f_{ZW}\ast K_{w,h}](z,w) = \int f_{ZW}(z,w')h_n^{-q}K_w\left(h_n^{-1}(w-w')\right)\dx w'$. Since $f_{ZW}\in C^s_M$, we obtain
	\begin{equation*}
	\left\|\E\hat r - r\right\| = O(h^s),
	\end{equation*}
	as shown in \cite{gine2015mathematical}, Proposition 4.3.8.  Consequently,
	\begin{equation*}
		\E\|\hat r - r\|^2 = O\left(\frac{1}{nh_n^q} + h_n^{2s}\right).
	\end{equation*}
	Next, decompose
	\begin{equation*}
		\begin{aligned}
			(\hat K\varphi_1 - K\varphi_1)(w) & \triangleq V_n(w) + B_n(w),
		\end{aligned}
	\end{equation*}
	where
	\begin{equation*}
		\begin{aligned}
			V_n & = \int\varphi_1(z)\left(\hat f_{ZW}(z,w) - \E\hat f_{ZW}(z,w)\right)\dx z, \\
			B_n & = \int\varphi_1(z)\left(\E\hat f_{ZW}(z,w) - f_{ZW}(z,w)\right)\dx z.
		\end{aligned}
	\end{equation*}
	By the Cauchy-Schwartz inequality,
	\begin{equation*}
		\|B_n\| \leq \|\varphi_1\|\left\|\E\hat f_{ZW} - f_{ZW}\right\|,
	\end{equation*}
	and the right side is of order $O(h_n^s)$ under the assumption $f_{ZW}\in C^s_M$; see \cite{gine2015mathematical}, p.404.
	
	Next, note that 
	\begin{equation*}
		V_n(w) = \frac{1}{nh_n^q}\sum_{i=1}^n\eta_{n,i}(w)
	\end{equation*}
	with 
	\begin{equation*}
		\eta_{n,i}(w) = K_w\left(h_n^{-1}(W_i - w)\right)\left[\varphi_1\ast K_z\right](Z_i) - \E\left[K\left(h_n^{-1}(W_i-w)\right)\left[\varphi_1\ast K_z\right](Z_i)\right],
	\end{equation*}
	where $\left[\varphi_1\ast K_z\right](Z_i) = \int \varphi_1(z)h_n^{-p}K_z\left(h_n^{-1}(Z_i - z)\right)\dx z$. Then
	\begin{equation*}
		\begin{aligned}
			\E\|V_n\|^2 & \leq \frac{1}{nh_n^{2q}}\int\int\int\left|K_w(h_n^{-1}(w' - w))\right|^2\left|\left[\varphi_1\ast K_z\right](z')\right|^2\dx wf_{ZW}(z',w')\dx w'\dx z'\\
			& = \frac{1}{nh_n^q}\|K_w\|^2\int \left|\left[\varphi_1\ast K_z\right](z)\right|^2 f_Z(z)\dx z \\
			& = O\left(\frac{1}{nh_n^q}\right),
		\end{aligned}
	\end{equation*}
	where the second line follows from a change of variables, and the last by $\|f_Z\|_\infty\leq C$ and Young's inequality. Combining all estimates, we establish the first claim.
	
	The second claim follows from
	\begin{equation*}
		\E\left\|\hat K - K\right\|^2 \leq \E\left\|\hat f_{ZW} - f_{ZW}\right\|^2,
	\end{equation*}
	and standard results on the $L_2$ error of kernel density estimators; see \cite{gine2015mathematical}, Chapter 5.
\end{proof}

\begin{proof}[Proof of Theorem~\ref{thm:inner_products_flir}]
	We first examine the distribution of $n\alpha_n\langle \hat\varphi - \varphi_1,\mu_0\rangle$. Define $b_n = \alpha_n(\alpha_n I + K^*K)^{-1}\varphi_1$ and observe that $(\alpha_nI+\hat K^*\hat K)^{-1}\hat K^*=\hat K^*(\alpha_nI+\hat K\hat K^*)^{-1}$. Then, as in the proof of Theorem~\ref{thm:L2_rate}, decompose
	\begin{equation*}
		\begin{aligned}
			\left\langle\hat\varphi - \varphi_1,\mu_0\right\rangle & = \left\langle\hat K^*(\alpha_nI + KK^*)^{-1}(\hat r - \hat K\varphi_1),\mu_0\right\rangle \\
			& \qquad + \left\langle\hat K^*\left((\alpha_nI + \hat K\hat K^*)^{-1} - (\alpha_nI + KK^*)^{-1}\right) (\hat r - \hat K\varphi_1),\mu_0\right\rangle \\
			& \qquad + \left\langle (\alpha_n I + \hat K^*\hat K)^{-1}\hat K^*(\hat K - K)b_n,\mu_0 \right\rangle \\
			& \qquad + \left\langle (\alpha_n I + \hat K^*\hat K)^{-1}(\hat K^* - K^*)Kb_n,\mu_0 \right\rangle \\
			& \qquad + \langle b_n,\mu_0\rangle \\
			& \triangleq I_n + II_n + III_n + IV_n + V_n.
		\end{aligned}
	\end{equation*}
	We will show that
	\begin{equation*}
		I_n = \left\langle \alpha_n(\alpha_n I + KK^*)^{-1}(\hat r - \hat K\varphi_1),\hat K\mu_0\right\rangle
	\end{equation*}
	is the leading term in this decomposition and that all other terms are asymptotically negligible. Since $\mu_0\in\mathcal{N}(K)$, we have
	\begin{equation}\label{eq:mu_0}
		(\alpha_nI + K^*K)^{-1}\mu_0 = \frac{1}{\alpha_n}\mu_0.
	\end{equation}
	Thus,
	\begin{equation*}
		\begin{aligned}
			II_n & = \left\langle (\alpha_n I + \hat K\hat K^*)^{-1}(KK^* - \hat K\hat K^*)(\alpha_nI + KK^*)^{-1}(\hat r - \hat K\varphi_1),\hat K\mu_0\right\rangle\\
			& = \left\langle(\alpha_n I + \hat K\hat K^*)^{-1}\hat K(K^*-\hat K^*)(\alpha_nI + KK^*)^{-1}(\hat r - \hat K\varphi_1),\hat K\mu_0 \right\rangle \\
			& \qquad + \left\langle(\alpha_n I + \hat K\hat K^*)^{-1}(K-\hat K)K^*(\alpha_nI + KK^*)^{-1}(\hat r - \hat K\varphi_1),\hat K\mu_0 \right\rangle. \\
		\end{aligned}
	\end{equation*}
	Using the Cauchy-Schwartz inequality and similar computations to those in the proof of Theorem~\ref{thm:L2_rate}
	\begin{equation*}
		\begin{aligned}
			II_n & \leq \left\|(\alpha_n I + \hat K\hat K^*)^{-1}\hat K\right\|\left\|K^* - \hat K^*\right\|\left\|(\alpha_nI+KK^*)^{-1}\right\|\left\|\hat r - \hat K\varphi_1\right\| \left\|\hat K\mu_0\right\| \\
			& \quad + \left\|(\alpha_n I + \hat K\hat K^*)^{-1}\right\|\left\|K - \hat K\right\|\left\|K^*(\alpha_nI+KK^*)^{-1}\right\|\left\|\hat r - \hat K\varphi_1\right\| \left\|(\hat K - K)\mu_0\right\| \\
			& \leq \frac{1}{\alpha_n^{3/2}}\left\|\hat K - K\right\|^2\left\|\hat r - \hat K\varphi_1\right\|\|\mu_0\|.
		\end{aligned}
	\end{equation*}	
	Next, under Assumption~\ref{as:source4} from the proof of Theorem~\ref{thm:L2_rate}, we also know that $\|b_n\| = O(\alpha_n^{(\beta/2)\wedge 1})$ and $\|Kb_n\| = O(\alpha_n^{((\beta+1)/2)\wedge 1})$, so
	\begin{equation*}
		\begin{aligned}
			III_n & \leq \left\|(\alpha_n I + \hat K^*\hat K)^{-1}\hat K^*\right\|\left\|\hat K- K\right\| \|b_n\| \|\mu_0\| \\
			& \lesssim \frac{1}{\alpha_n^{1/2}}\left\|\hat K - K\right\| \alpha_n^{\frac{\beta}{2}\wedge 1}
		\end{aligned}
	\end{equation*}
	and
	\begin{equation*}
		\begin{aligned}
			IV_n & \leq \left\|(\alpha_n I + \hat K^*\hat K)^{-1}\right\|\left\|\hat K^* - K^*\right\|\|Kb_n\|\|\mu_0\| \\
			& \lesssim \frac{1}{\alpha_n^{1/2}}\left\|\hat K - K\right\|\alpha_n^{\frac{\beta\wedge 1}{2}}.
		\end{aligned}
	\end{equation*}	
	Lastly, the bias term $\langle b_n,\mu_0\rangle$ is zero due to equation~(\ref{eq:mu_0}) and the orthogonality between $\varphi_1$ and $\mu_0$
	\begin{equation*}
		\langle b_n,\mu_0\rangle = \left\langle\varphi_1,\alpha_n(\alpha_nI+K^*K)^{-1}\mu_0\right\rangle = \langle \varphi_1,\mu_0\rangle = 0.
	\end{equation*}
	It follows from the discussion in Section~\ref{sec:rates} under Assumption~\ref{as:data4} that
	\begin{equation*}
		\left\|\hat K - K\right\| = O_P\left(\frac{1}{n^{1/2}}\right)\quad \text{and}\quad \left\|\hat r - \hat K\varphi_1\right\| = O_P\left(\frac{1}{n^{1/2}}\right).
	\end{equation*}
	Thus, since $n\alpha_n^{1 + \beta\wedge 1}\to 0$ and $n\alpha_n\to\infty$ under Assumption~\ref{as:tuning_inner},  
	\begin{equation*}
		\begin{aligned}
			n\alpha_n\langle\hat{\varphi} - \varphi_1,\mu_0\rangle & = n\alpha_n\left\langle(\alpha_n I + KK^*)^{-1}(\hat r - \hat K\varphi_1),\hat K\mu_0 \right\rangle + o_P(1) \\
			& \triangleq S_n + o_P(1),
		\end{aligned}
	\end{equation*}
	where
	\begin{equation*}
		S_n = n\alpha_n\left\langle(\alpha_n I + KK^*)^{-1}(\hat r - \hat K\varphi_1),\hat K\mu_0\right\rangle.
	\end{equation*}
	Next, decompose $S_n = S_n^0+S_n^1$ with
	\begin{equation*}
		\begin{aligned}
			S_n^0  & = \frac{1}{n}\sum_{i,j=1}^n(Y_i - \langle Z_i,\varphi_1))\left\langle Z_j, \mu_0\right\rangle\left\langle\alpha_n(\alpha_n I + KK^*)^{-1}W_i^0,W_j^0\right\rangle, \\
			S_n^1 & = \frac{1}{n}\sum_{i,j=1}^n(Y_i - \langle Z_i,\varphi_1\rangle)\langle Z_j,\mu_0\rangle \left\langle\alpha_n(\alpha_nI + KK^*)^{-1}W_i^1,W_j^1 \right\rangle.
		\end{aligned}
	\end{equation*}
	Since $W_i^0\in\mathcal{N}(K^*)$, we have $(\alpha_nI+KK^*)^{-1}W_i^0=\frac{1}{\alpha_n}W_i^0$. Using this fact, decompose further $S_n^0 \triangleq \zeta_n^0 + \mathbf{U}_n^0$ with
	\begin{equation*}
		\begin{aligned}
			\zeta_n^0 & = \frac{1}{n}\sum_{i=1}^n(Y_i - \langle Z_i,\varphi_1\rangle)\langle Z_i,\mu_0\rangle\left\|W_i^0\right\|^2, \\
			\mathbf{U}_n^0 & = \frac{1}{n}\sum_{i<j}\left\{\langle Z_i,\mu_0\rangle (Y_j - \langle Z_j,\varphi_0\rangle) + \langle Z_j,\mu_0\rangle(Y_i - \langle Z_i,\mu_0\rangle) \right\}\left\langle W_i^0,W_j^0\right\rangle.
		\end{aligned}
	\end{equation*}	
	Under Assumption~\ref{as:data4} by the strong law of large numbers
	\begin{equation*}
		\zeta_n^0 \xrightarrow{a.s.} \E\left[\left\|W^0\right\|^2 (Y - \langle Z,\varphi_1\rangle)\langle Z,\mu_0\rangle\right].
	\end{equation*}
	Next, note that $W^0=P_0W$ and $W^1=(I-P_0)W$, where $P_0$ is the projection operator on $\mathcal{N}(K^*)$. Since projection is a bounded linear operator, it commutes with the expectation, cf., \cite{bosq2000linear}, p.29, whence $\E\left[W^0\langle Z,\mu_0\rangle\right] = P_0K\mu_0 = 0$ and $\E\left[W^0(Y - \langle Z,\varphi_1\rangle )\right] = P_0\E[WU] + P_0K\varphi_0 = 0$. Therefore, $\mathbf{U}_n^0$ is a centered degenerate {$U$-statistic} with the kernel function $h$. Under Assumption~\ref{as:data4} by the {CLT for degenerate $U$-statistics}, see \cite{gregory1977large},
	\begin{equation*}
		\mathbf{U}_n^0 \xrightarrow{d} \sum_{j=1}^\infty \lambda_j(\chi_{j}^2 - 1).
	\end{equation*}	
	It remains to show that $S_n^1=o_P(1)$. To that end decompose $S_n^1 = \zeta_n^1 + \mathbf{U}_n^1$ with
	\begin{equation*}
		\begin{aligned}
			\zeta_n^1 & = \frac{1}{n}\sum_{i=1}^n(Y_i - \langle Z_i,\varphi_1\rangle)\langle Z_i,\mu_0\rangle\left\langle \alpha_n(\alpha_nI + KK^*)^{-1}W_i^1,W_i^1\right\rangle, \\
			\mathbf{U}_n^1 & = \frac{1}{n}\sum_{i\ne j}(Y_i - \langle Z_i,\varphi_1\rangle)\langle Z_j,\mu_0\rangle \left\langle\alpha_n(\alpha_nI + KK^*)^{-1}W_i^1,W_j^1\right\rangle.
		\end{aligned}
	\end{equation*}
	It follows from \cite{bakushinskii1967general} that $\left\|\alpha_n(\alpha_nI + KK^*)^{-1}W^1\right\| = o(1)$. Then under Assumption~\ref{as:data4} by the dominated convergence theorem
	\begin{equation*}
		\begin{aligned}
			\E\left|\zeta_n^1\right| & \leq \|\mu_0\|\E\left[\left(|U| + \|Z\|\|\varphi_0\|\right)\|Z\| \left\|\alpha_n(\alpha_nI + KK^*)^{-1}W^1\right\|\right] \\
			& \lesssim  \E\left[(\|UZ\| + \|Z\|^2)\|W\|\left\|\alpha_n(\alpha_nI + KK^*)^{-1}W^1\right\| \right] \\
			& = o(1),
		\end{aligned}
	\end{equation*}
	whence by Markov's inequality $\zeta_n^1=o_P(1)$. Lastly, note that
	\begin{equation*}
		\begin{aligned}
			\mathbf{U}_n^1 = \frac{1}{n}\sum_{i<j} & \left\{(Y_i - \langle Z_i,\varphi_1\rangle)\langle Z_j,\mu_0\rangle \left\langle\alpha_n(\alpha_nI + KK^*)^{-1}W_i^1,W_j^1\right\rangle\right. \\
			& \left.+ (Y_j - \langle Z_j,\varphi_1\rangle)\langle Z_i,\mu_0\rangle \left\langle\alpha_n(\alpha_nI + KK^*)^{-1}W_j^1,W_i^1\right\rangle\right\}
		\end{aligned}
	\end{equation*}
	is a centered degenerate {U-statistic}. Then by the moment inequality in \cite{korolyuk2013theory}, Theorem 2.1.3,
	\begin{equation*}
		\begin{aligned}
			\E\left|\mathbf{U}_n^1\right|^2 & \leq 2^{-1}\E\left|(U_1+\langle Z_1,\varphi_0\rangle)\langle Z_2,\mu_0\rangle \left\langle\alpha_n(\alpha_nI + KK^*)^{-1}W_1^1,W_2^1\right\rangle\right|^2 \\
			& \quad + 2^{-1}\E\left|(U_2+\langle Z_2,\varphi_0\rangle)\langle Z_1,\mu_0\rangle \left\langle\alpha_n(\alpha_nI + KK^*)^{-1}W_2^1,W_1^1\right\rangle\right|^2 \\
			& \lesssim \E\left[\|Z\|^2\left\|\alpha_n(\alpha_nI + KK^*)^{-1}W^1\right\|^2\right] = o(1),
		\end{aligned}
	\end{equation*}
	where the last line follows under Assumptions~\ref{as:data4} and \ref{as:tuning_inner}, and previous discussions. This shows that
	\begin{equation*}
		n\alpha_n\langle \hat\varphi - \varphi_1,\mu_0\rangle \xrightarrow{d} \E\left[\|W\|^2(Y - \langle Z,\varphi_1\rangle)\langle Z,\mu_0\rangle\right] + \sum_{j\geq 1}\lambda_j(\chi_{j}^2-1).
	\end{equation*}
	
	Next, we focus on the distribution of $\pi_n\langle\hat\varphi - \varphi_1,\mu_1\rangle$. Decompose
	\begin{equation*}
		\begin{aligned}
			\left\langle\hat\varphi - \varphi_1,\mu_1\right\rangle & = \left\langle(\alpha_nI + K^*K)^{-1}K^*(\hat r - \hat K\varphi_1),\mu_1\right\rangle \\
			& \qquad + \left\langle\left\{(\alpha_nI + \hat K^*\hat K)^{-1} - (\alpha_nI + K^*K)^{-1}\right\}\hat K^*(\hat r - \hat K\varphi_1),\mu_1\right\rangle \\
			& \qquad + \left\langle(\alpha_nI + K^*K)^{-1}(\hat K^* - K^*)(\hat r - \hat K\varphi_1),\mu_1\right\rangle \\
			& \qquad + \left\langle(\alpha_nI + \hat K^*\hat K)^{-1}\hat K^*\hat K\varphi_1 - (\alpha_nI + K^*K)^{-1}K^*K\varphi_1,\mu_1\right\rangle \\
			& \qquad + \langle b_n,\mu_1\rangle \\
			& \triangleq I_n' + II_n' + III_n'+ IV_n' + \langle b_n,\mu_1\rangle.
		\end{aligned}
	\end{equation*}	
	Put $\eta_{n} = (Y-\langle Z,\varphi_1\rangle)(\alpha_nI + K^*K)^{-1}K^*W$ and note that
	\begin{equation*}
		\begin{aligned}
			\Var(\langle\eta_n,\mu_1\rangle) & = \E\left|\langle (Y-\langle Z,\varphi_1)W,K(\alpha_n I + K^*K)^{-1}\mu_1\rangle\right|^2 \\
			& = \langle\Sigma K(\alpha_n I + K^*K)^{-1}\mu_1,K(\alpha_n I + K^*K)^{-1}\mu_1 \rangle \\
			& = \|\Sigma^{1/2}K(\alpha_n I + K^*K)^{-1}\mu_1\|^2.
		\end{aligned}
	\end{equation*}
	Suppose that for all $\epsilon>0$, the following Lindeberg's condition is satisfied
	\begin{equation}\label{eq:lindeberg}
		\lim_{n\to\infty}\frac{\pi_n^2}{n}\E\left[\left|\langle\eta_{n},\mu_1\rangle\right|^2\one_{\left\{\pi_n\left|\langle\eta_{n},\mu_1\rangle\right|\geq \epsilon n\right\}}\right] = 0
	\end{equation}
	with $\pi_n = n^{1/2}\left\|\Sigma^{1/2}K(\alpha_n I + K^*K)^{-1}\mu_1\right\|^{-1}$. Then by the Lindeberg-Feller central limit theorem 
	\begin{equation*}
		\begin{aligned}
			\pi_nI_n' & = \frac{\pi_n}{n}\sum_{i=1}^n(U_i + \langle Z_i,\varphi_0\rangle)\left\langle(\alpha_nI + K^*K)^{-1}K^*W_i,\mu_1\right\rangle \\
			& \xrightarrow{d} N(0,1).
		\end{aligned}
	\end{equation*}
	and the result follows provided that all other terms are asymptotically negligible. To see that Lindeberg's condition in equation~(\ref{eq:lindeberg}) is satisfied, note that for every $\delta>0$,
	\begin{equation*}
		\E\left[\left|\langle\eta_{n},\mu_1\rangle\right|^2\one_{\left\{\pi_n\left|\langle\eta_{n},\mu_1\rangle\right|\geq \epsilon n\right\}}\right] \leq \frac{\pi_n^\delta}{\epsilon^\delta n^{\delta}}\E\left|\langle\eta_{n},\mu_1\rangle\right|^{2+\delta}
	\end{equation*}
	and that $\pi_n\sim n^{-c}$ with $c\in(0,1/2]$ depending on the mapping properties of operators $K$ and $\Sigma$. Therefore, the Lindeberg condition is satisfied provided that $\E\left|\langle\eta_{n},\mu_1\rangle\right|^{2+\delta} = O(1)$. This is easily verified under Assumption~\ref{as:data4} since
	\begin{equation*}
		\begin{aligned}
			\left|\langle\eta_{n},\mu_1\rangle\right| & \lesssim \left|U + \langle Z,\varphi_0\rangle\right| \left\|(K^*K)^{\tilde \gamma}(\alpha_n I + K^*K)^{-1}K^*(K^*K)^{\gamma}\right\| \\
			& \lesssim |U| + |\langle Z,\varphi_0\rangle|.
		\end{aligned}
	\end{equation*}
	
	Therefore, it remains to show that all other terms normalized with $\pi_n$ are asymptotically negligible. For $II_n'$, by the Cauchy-Schwartz inequality
	\begin{equation*}
		\begin{aligned}
			II_n' & = \left\langle\frac{1}{n}\sum_{i=1}^nW_i(U_i + \langle Z_i,\varphi_0\rangle),\hat K^*\left((\alpha_nI + \hat K^*\hat K)^{-1} - (\alpha_nI + K^*K)^{-1}\right)\mu_1\right\rangle \\
			& \leq \left\|\frac{1}{n}\sum_{i=1}^nW_i(U_i+\langle Z_i,\varphi_0\rangle)\right\|\left\|\hat K^*(\alpha_nI + \hat K^*\hat K)^{-1}(\hat K^*\hat K - K^*K)(\alpha_nI + K^*K)^{-1}\mu_1\right\|.
		\end{aligned}
	\end{equation*}
	
	Since $\mu_1\in\mathcal{R}\left[(K^*K)^{\gamma}\right]$, there exists some $\psi\in \mathcal{E}$ such that $\mu_1 = (K^*K)^{\gamma}\psi$ and so
	\begin{equation*}
		\begin{aligned}
			II_n' & \lesssim_P n^{-1/2}\left\|\hat K^*(\alpha_nI + \hat K^*\hat K)^{-1}\hat K^*\right\|\left\|\hat K - K\right\|\left\|(\alpha_nI + K^*K)^{-1}(K^*K)^\gamma\psi\right\| \\
			& \quad + n^{-1/2}\left\|\hat K^*(\alpha_nI + \hat K^*\hat K)^{-1}\right\|\left\|\hat K^* - K^*\right\|\left\|K(\alpha_nI + K^*K)^{-1}(K^*K)^\gamma\psi\right\| \\
			& \lesssim_P n^{-1}\left\|(\alpha_nI + K^*K)^{-1}(K^*K)^\gamma\psi\right\| + n^{-1}\alpha_n^{-1/2}\left\|K(\alpha_nI + K^*K)^{-1}(K^*K)^\gamma\psi\right\| \\
			& \lesssim_P n^{-1}\alpha_n^{\gamma\wedge 1 - 1} + n^{-1}\alpha_n^{\gamma\wedge 1/2-1} = o_P(\pi_n^{-1}),
		\end{aligned}
	\end{equation*}
	where the last line follows under Assumption~\ref{as:tuning_inner}. Similarly,
	\begin{equation*}
		\begin{aligned}
			III_n' & \leq \left\|\hat K^* - K^*\right\|\left\|\frac{1}{n}\sum_{i=1}^nW_i(U_i+\langle Z_i,\varphi_0\rangle)\right\|\left\|(\alpha_n I + K^*K)^{-1}(K^*K)^\gamma\psi\right\| \\
			& \lesssim_P n^{-1}\alpha_n^{\gamma\wedge 1 -1} \\
			& = o_P(\pi_n^{-1}).
		\end{aligned}
	\end{equation*}
	Next, decompose
	\begin{equation*}
		\begin{aligned}
			IV_n' & = \left\langle(\alpha_nI + \hat K^*\hat K)^{-1}\hat K^*\hat K\varphi_1 - (\alpha_nI + K^*K)^{-1}K^*K\varphi_1,\mu_1\right\rangle \\
			& = \left\langle \alpha_n(\alpha_nI + \hat K^*\hat K)^{-1}\left[\hat K^*\hat K - K^*K\right](\alpha_nI + K^*K)^{-1}\varphi_1,\mu_1\right\rangle \\
			&  = \left\langle(\alpha_nI + \hat K^*\hat K)^{-1}\hat K^*(\hat K - K)b_n,\mu_1\right\rangle \\
			& \qquad + \left\langle(\alpha_nI + \hat K^*\hat K)^{-1}(\hat K^* - K^*)Kb_n,\mu_1\right\rangle \\
			& \triangleq IV_n^a + IV_n^b + IV_n^c + IV_n^d + IV_n^e
		\end{aligned}
	\end{equation*}	
	with 
	\begin{equation*}
		\begin{aligned}
			IV_n^a & = \left\langle \left\{(\alpha_n I + \hat K^*\hat K)^{-1} - (\alpha_n I + K^*K)^{-1}\right\}\hat K^*(\hat K - K)b_n,\mu_1 \right\rangle \\
			IV_n^b & = \left\langle \left\{(\alpha_n I + \hat K^*\hat K)^{-1} - (\alpha_n I + K^*K)^{-1}\right\}(\hat K^* - K^*)Kb_n,\mu_1 \right\rangle \\
			IV_n^c & = \left\langle (\alpha_n I + K^*K)^{-1}K^*(\hat K - K)b_n,\mu_1 \right\rangle \\
			IV_n^d & = \left\langle (\alpha_n I + K^*K)^{-1}(\hat K^* - K^*)(\hat K - K)b_n,\mu_1 \right\rangle\\
			IV_n^e & = \left\langle (\alpha_n I + K^*K)^{-1}(\hat K^* - K^*)Kb_n,\mu_1 \right\rangle.
		\end{aligned}
	\end{equation*}
	We bound the last three terms by the Cauchy-Schwartz inequality
	\begin{equation*}
		\begin{aligned}
			IV_n^c & \leq \left\|\hat K - K\right\| \|b_n\| \left\|K(\alpha_n I + K^*K)^{-1}\mu_1\right\| \lesssim_P \frac{\alpha_n^{\frac{\beta}{2}\wedge1+\gamma\wedge \frac{1}{2}}}{\sqrt{n\alpha_n}} \\
			IV_n^d & \leq \left\|\hat K^* - K^*\right\|\left\|\hat K - K\right\|\|b_n\|\left\|(\alpha_n I + K^*K)^{-1}\mu_1\right\| \lesssim_P\frac{\alpha_n^{\frac{\beta}{2}\wedge 1 + \gamma\wedge 1}}{n\alpha_n} \\
			IV_n^e & \leq \left\|\hat K^* - K^*\right\|\|Kb_n\|\left\|(\alpha_n I + K^*K)^{-1}\mu_1\right\| \lesssim_P \frac{\alpha_n^{\frac{\beta}{2}\wedge \frac{1}{2} + \gamma\wedge 1}}{\sqrt{n\alpha_n}}.
		\end{aligned}
	\end{equation*}
	Next, for the first two terms, by the Cauchy-Schwartz inequality, we have
	\begin{equation*}
		\begin{aligned}
			IV_n^a &\leq \left\|\hat K - K\right\|^2\left\|(\alpha_nI + \hat K^*\hat K)^{-1}\hat K^*\right\|\|b_n\|\left\|K(\alpha_n I + K^*K)^{-1}\mu_1\right\| \\
			&\quad + \left\|\hat K^* - K^*\right\| \left\|\hat K(\alpha_n I + \hat K^*\hat K)^{-1}\hat K^*\right\| \left\|\hat K - K\right\| \|b_n\| \left\|(\alpha_n I+  K^*K)^{-1}\mu_1\right\|\\
			& \lesssim_P \frac{\alpha_n^{\frac{\beta}{2}\wedge 1 + \gamma\wedge \frac{1}{2}}}{n\alpha_n} + \frac{\alpha_n^{\frac{\beta}{2}\wedge1 + \gamma\wedge 1}}{n\alpha_n} \lesssim_P \frac{\alpha_n^{\frac{\beta}{2}\wedge 1 + \gamma\wedge \frac{1}{2}}}{n\alpha_n}
		\end{aligned}
	\end{equation*}
	and
	\begin{equation*}
		\begin{aligned}
			IV_n^b &\leq \left\|\hat K - K\right\|\left\|(\alpha_n I + \hat K^*\hat K)^{-1}\right\|\left\|\hat K^* - K^*\right\|\|Kb_n\|\left\|K(\alpha_n I + K^*K)^{-1}\mu_1\right\| \\
			& \quad + \left\|\hat K^* - K^*\right\|^2\left\|\hat K(\alpha_n I + \hat K^*\hat K)^{-1}\right\|\|Kb_n\|\left\|(\alpha_n I + K^*K)^{-1}\mu_1\right\| \\
			& \lesssim_P \frac{\alpha_n^{\frac{\beta}{2}\wedge\frac{1}{2}+\gamma\wedge\frac{1}{2}}}{n\alpha_n} + \frac{\alpha_n^{\frac{\beta}{2}\wedge\frac{1}{2} + \gamma\wedge 1}}{n\alpha_n} \lesssim_P \frac{\alpha_n^{\frac{\beta}{2}\wedge\frac{1}{2}+\gamma\wedge\frac{1}{2}}}{n\alpha_n}.
		\end{aligned}
	\end{equation*}
	Lastly,
	\begin{equation*}
		\pi_n\langle b_n,\mu_1\rangle \lesssim \pi_n \|(K^*K)^{\gamma}b_n\| \lesssim \pi_n\alpha_n^{(\gamma+\beta/2)\wedge 1}.
	\end{equation*}
	This completes the proof.
\end{proof}

\section{Partial Identification}\label{sec:pi}
In this section, we discuss how our results can be applied to partial identification, drawing on ideas from \cite{freyberger2017completeness} and \cite{santos2012inference}. This section is intended primarily for illustrative purposes and a detailed exploration of the partial identification approach is left for future research.

Define the $L_\infty$-diameter of the identified set as
\begin{equation*}
	\left|I_0\right| \triangleq \sup_{\phi_1,\phi_2\in I_0}\|\phi_1 - \phi_2\|_\infty.
\end{equation*}
Suppose that both $\varphi$ and $\varphi_1$ belong to the Lipschitz smoothness class
\begin{equation*}
	\mathscr{F} \triangleq \left\{\phi\in L_\infty([0,1]^p):\;\|\phi\|_s \triangleq \|\phi\|_\infty + \sup_{z_1\ne z_2}\frac{|\phi(z_1) - \phi(z_2)|}{\|z_1-z_2\|}\leq C \right\}.
\end{equation*}

If the best approximation $\varphi_1$ and the diameter of the identified set $|I_0|$ are known, then $\varphi_1\pm|I_0|$ is a valid identified set for the structural function $\varphi$. Our results establish that we can estimate the best approximation consistently. The remaining task is to infer the diameter of the identified set.

To this end, for a fixed $\varepsilon>0$, we consider the following hypotheses:
\begin{equation*}
	H_0:\;\left|I_0\right|\geq \varepsilon\qquad \textrm{vs.} \qquad H_1:\;\left|I_0\right|<\varepsilon
\end{equation*}
following  the approach in \cite{freyberger2017completeness} to construct {a test statistic}. Under $H_0$, there exist functions $\phi_1,\phi_2\in I_0$ such that $\|\phi_1-\phi_2\|_\infty\geq \varepsilon$. For $\phi=\phi_1-\phi_2$, we have  $K\phi = 0$, $\|\phi\|_s\leq 2C$, and $\|\phi\|_\infty\geq \varepsilon$. This motivates the following {statistic}
\begin{equation*}
	T_\varepsilon\triangleq \inf_{\phi:\|\phi\|_s\leq 2C, \|\phi\|_\infty\geq \varepsilon}\|K\phi\|^2 = \inf_{\phi\in\mathscr{F}_\varepsilon, \|\phi\|_\infty=1}\|K\phi\|^2.
\end{equation*}
where $\mathscr{F}_\varepsilon = \left\{\phi\in L_\infty([0,1]^p):\; \|\phi\|_s \leq 2C/\varepsilon \right\}$. Clearly, under $H_0$, we have $T_\varepsilon=0$. Define the sample counterpart as
\begin{equation*}
	\hat T_\varepsilon = \inf_{\phi\in\mathscr{F}_\varepsilon:\; \|\phi\|_\infty=1}\|\hat K\phi\|^2,
\end{equation*}
where $\hat K$ is a consistent estimator of $K$. 

For NPIV regression, suppose that the density $\hat f_{ZW}$ is estimated using a series estimator, and the { statistic} $\hat T_\varepsilon$ is computed over a $J$-dimensional sieve. Theorem 2 in \cite{freyberger2017completeness} shows that
\begin{equation*}
	\hat\varepsilon = \sup\left\{\varepsilon\in[0,\bar C]:\; n\hat T_\varepsilon \leq q_{1-\alpha}\right\}
\end{equation*}
consistently estimates the upper bound on $|I_0|$, where $\bar C$ is the largest $\varepsilon$ for which $\{\phi:\; \|\phi\|_s\leq 2C,\|\phi\|_\infty\geq \varepsilon \}\ne \emptyset$ and $q_{1-\alpha}$ is a critical value.

Combining this result with our result on the $L_\infty$ consistency for the best approximation, the interval $[\hat\varphi-\hat\varepsilon,\hat\varphi+\hat\varepsilon]$ serves as a valid set estimator for $\varphi$.\footnote{We thank the anonymous referee for highlighting this connection.}

\section{Extreme Nonidentification}\label{sec:extreme}
In this section, we derive an approximation of the large sample distribution of the Tikhonov-regularized estimators in cases of extreme nonidentification. Remarkably, we show that the asymptotic distribution is a weighted sum of independent chi-squared random variables. This result serves as a foundation for Section~\ref{sec:flir}, where we analyze a transition between chi-squared and Gaussian limits in intermediate cases.\footnote{Extreme nonidentification also relates to the weak instruments problem.}

\subsection{High-dimensional Regressions}
In high-dimensional regressions, the strength of identification is determined by the covariance operator of $Z$ and $W$. In cases of extreme nonidentification, this operator is degenerate, leading to the following result.
\begin{theorem}\label{thm:u_statistic}
	Suppose that Assumption~\ref{as:data4} holds, $\E[\langle Z,\delta\rangle W]=0,\;\forall \delta\in \mathcal{E}$, and $\alpha_nn\to\infty$. Then
	\begin{equation*}
		\alpha_nn(\hat \varphi - \varphi_1)\xrightarrow{d} \E\left[\|W\|^2YZ\right] + J(h),
	\end{equation*}
	where $h(X,X') = \frac{1}{2}\langle W,W'\rangle(ZY' + Z'Y)$, $X'=(Y',Z',W')$ is an independent copy of $X=(Y,Z,W)$, and $J$ is a stochastic Wiener-It\^{o} integral.
\end{theorem}
The theorem states weak convergence in the topology of the Hilbert space $\mathcal{E}$, which cannot be achieved in cases of regular identification. Furthermore, it can be shown that the distribution of the inner products of $J(h)$ with $\mu\in\mathcal{E}$ is a weighted sum of chi-squared random variables. Interestingly, Theorem~\ref{thm:u_statistic} does not require that $\alpha_n\to 0$ as $n\to\infty$.

\begin{proof}[Proof of Theorem~\ref{thm:u_statistic}]
	Given that $\E[\langle Z,\delta\rangle W]=0$ for all $\delta\in\mathcal{E}$, we have $\varphi_1 = 0$. Thus, 
	\begin{equation*}
		\alpha_nn\left(\hat{\varphi} - \varphi_1\right) = \left(I + \frac{1}{\alpha_n}\hat K^*\hat K\right)^{-1}n\hat K^*\hat r.
	\end{equation*}
	Under Assumption~\ref{as:data4}
	\begin{equation*}
		\begin{aligned}
			\E\|\hat K\|^2 = \E\|\hat K - K\|^2 \leq \E\left\|\frac{1}{n}\sum_{i=1}^nZ_iW_i - \E[ZW]\right\|^2 = O\left(\frac{1}{n}\right).
		\end{aligned}
	\end{equation*}
	Thus, $\|\hat K^* \hat K\| \leq \|\hat K\|^2 = O_P\left(n^{-1}\right)$. Since $\alpha_nn\to\infty$, the continuous mapping theorem (see \cite{van2000weak}, Theorem 1.3.6) implies that
	\begin{equation*}
		\alpha_nn\left(\hat{\varphi} - \varphi_1\right) = (I + o_P(1))^{-1}n\hat K^*\hat r.
	\end{equation*}
	Applying Slutsky's theorem (see \cite{van2000weak}, Example 1.4.7), it suffices to determine the asymptotic distribution of $n\hat K^*\hat r$.
	
	Observe that
	\begin{equation*}
		\begin{aligned}
			n\hat K^*\hat r & = \frac{1}{n}\sum_{i,j=1}^n\left\langle W_i,W_j \right\rangle Z_iY_j \\
			& = \frac{1}{n}\sum_{i=1}^n \|W_i\|^2Z_iY_i + \frac{1}{n}\sum_{i\ne j}\left\langle W_i,W_j \right\rangle Z_iY_j. \\
		\end{aligned}
	\end{equation*}
	Under Assumption~\ref{as:data4}, by the Mourier law of large numbers
	\begin{equation*}
		\frac{1}{n}\sum_{i=1}^n \|W_i\|^2Z_iY_i \xrightarrow{a.s.} \E\left[\|W\|^2ZY\right].
	\end{equation*}
	
	Since $\E[\langle Z,\delta\rangle W]=0,\forall \delta\in\mathcal{E}$, the second term forms a Hilbert space-valued degenerate {$U$-statistic}
	\begin{equation*}
		\begin{aligned}
			n\mathbf{U}_n & \triangleq  \frac{1}{n}\sum_{i\ne j}\left\langle W_i,W_j \right\rangle Z_iY_j \\
			& = \frac{2}{n}\sum_{i<j}\frac{Z_iY_j + Z_jY_i}{2}\left\langle W_i,W_j \right\rangle.
		\end{aligned}
	\end{equation*}
	Under the Assumption~\ref{as:data4}, the Borovskich CLT, see Online Supplementary Material, Theorem~\ref{thm:borovskikh}, gives
	\begin{equation*}
		n\mathbf{U}_n \xrightarrow{d} J(h),
	\end{equation*}
	where $J(h)=\iint_{\mathcal{X}\times\mathcal{X}} h(x_1,x_2)\mathbb{W}(\dx x_1)\mathbb{W}(\dx x_2)$ is a stochastic Wiener-It\^{o} integral, $\mathbb{W}$ is a Gaussian random measure on $\mathcal{X}$, $h(X,X') = \frac{ZY' + Z'Y}{2}\langle W,W'\rangle$, and $X'=(Y',Z',W')$ is an independent copy of $X=(Y,Z,W)$.
\end{proof}

\subsection{Nonparametric IV Regression}
In nonparametric IV regression, the  strength of identification is described by the conditional expectation operator. In the extreme nonidentified case,
\begin{equation*}
	\E[\phi(Z)|W]=0,\qquad \forall \phi\in L_{2,0}(Z),
\end{equation*}
where $L_{2,0}(Z)=\left\{\phi\in L_2(Z):\; \E\phi(Z)=0 \right\}$, and $K$ is a degenerate conditional expectation operator. Consider the operator $T:\phi\mapsto \E_{X}[\phi(X)h(X,X')]$ on $L_2(X)$, where $\E_X$ denotes expectation with respect to $X=(Y,Z,W)$ only, $X'$ is an independent copy of $X$, and
\begin{equation*}
	h(x,x')=\frac{1}{2}\left\{yP_0\mu(z') + y'P_0\mu(z)\right\}h_w^{-q}\bar K\left(h_w^{-1}(w-w')\right).
\end{equation*}
Here, $P_0$ is the projection operator on $L_{2,0}$ and $\bar K(v)=\int K_w(v-u)K_w(u)\dx u$ is the convolution kernel. We assume the following mild conditions on the distribution of the data:

\begin{assumption}\label{as:degenerate_npiv}
	(i) $(Y_i,Z_i,W_i)_{i=1}^n$ is an i.i.d. sample of $(Y,Z,W)$; (ii) $\E\left[|Y||Z\right]<\infty$, $\E[|Y|^2|W]<\infty$ a.s.; (iii) $K_j\in L_1\cap L_2,j\in\{z,w\}$ and $K_w$ is a symmetric and bounded function; (iv) $f_Z\in L_\infty$.
\end{assumption}

Let $h_z$ and $h_w$ be the bandwidth parameters for smoothing over $Z$ and $W$, respectively. The following result holds:
\begin{theorem}\label{thm:degenerate_npiv}
	Suppose that Assumption~\ref{as:degenerate_npiv} holds, $\E[\phi(Z)|W]=0,\forall \phi\in L_{2,0}(Z)$, and $n\alpha_nh_z^p\to \infty$ with $h_w$ fixed. Then for every $\mu\in L_2([0,1]^p)$,
	\begin{equation*}
		\alpha_nn \langle \hat \varphi - \varphi_1,\mu\rangle \xrightarrow{d} \E\left[YP_0\mu(Z)\right]h_w^{-q}\bar K(0) + \sum_{j=1}^\infty\lambda_j(\chi_{j}^2 - 1),
	\end{equation*}
	where $(\chi^2_{j})_{j\geq 1}$ are independent chi-squared random variables with 1 degree of freedom and $(\lambda_j)_{j\geq 1}$ are eigenvalues of $T$.
\end{theorem}
\begin{proof}[Proof of Theorem~\ref{thm:degenerate_npiv}]
	Since $\E[\phi(Z)|W]=0,\forall \phi\in L_{2,0}(Z)$, we find that $\varphi_1=0$. Additionally, the adjoint operator of $K$ is $P_0K^*$, where $P_0$ is the orthogonal projection on $L_{2,0}(Z)$. Thus,
	\begin{equation*}
		\begin{aligned}
			\alpha_nn(\hat{\varphi} - \varphi_1) = \left(I + \frac{1}{\alpha_n}P_0\hat K^*\hat K\right)^{-1}nP_0\hat K^*\hat r,
		\end{aligned}
	\end{equation*}
	where $\hat P_0$ is the estimator of $P_0$. Since $\E[\phi(Z)|W]=0$ for all $\phi\in L_{2,0}(Z)$, under Assumption~\ref{as:degenerate_npiv} (i),
	\begin{equation*}
		\begin{aligned}
			\E\left\|P_0\hat K^*\hat K\right\| & \leq \E\|P_0\hat K\|^2  \leq \E\| P_0\hat f_{ZW}\|^2 \\
			& = \E\left\|\frac{1}{nh_z^ph_w^q}\sum_{i=1}^nP_0K_z\left(h_z^{-1}(Z_i - z)\right)K_w\left(h_z^{-1}(W_i - w)\right)\right\|^2 \\
			& \leq \frac{1}{nh_z^{2p}h_w^{2q}}\E\left\|P_0K_z\left(h_z^{-1}(Z_i - z)\right)K_w\left(h_z^{-1}(W_i - w)\right)\right\|^2 \\
			& = \frac{1}{nh_z^ph_w^q}\|P_0K_z\|\|K_w\| = O\left(\frac{1}{nh_z^p}\right).
		\end{aligned}
	\end{equation*}
	Therefore, as $n\alpha_nh_z^p\to \infty$, we have $\frac{1}{\alpha_n}\left\|\hat P_0\hat K^*\hat K\right\| = o_P(1)$. By the continuous mapping and Slutsky's theorems, it then suffices to determine the asymptotic distribution of
	\begin{equation*}
		nP_0\hat K^*\hat r = \frac{1}{nh_z^ph_w^q}\sum_{i,j}Y_iP_0K_z\left(h_z^{-1}(Z_j-z)\right)\bar K\left(h_w^{-1}(W_i-W_j)\right).
	\end{equation*}
	For every $\mu\in L_2([0,1]^p)$, we have
	\begin{equation*}
		\begin{aligned}
			\left\langle nP_0\hat K^*\hat r,\mu\right\rangle & = \left\langle n\hat K^*\hat r,P_0\mu\right\rangle \\
			& \triangleq \zeta_n + \mathbf{U}_n + R_n,
		\end{aligned}
	\end{equation*}
	where
	\begin{equation*}
		\begin{aligned}
			\zeta_n & = \frac{1}{n}\sum_{i=1}^nY_iP_0\mu(Z_i)h_w^{-q}\bar K(0), \\
			\mathbf{U}_n & = \frac{2}{n}\sum_{i<j}\frac{1}{2}\left\{Y_iP_0\mu(Z_j) + Y_jP_0\mu(Z_i)\right\}h_w^{-q}\bar K\left(h_w^{-1}(W_i-W_j)\right), \\
			R_n & = \frac{1}{nh_w^q}\sum_{i,j=1}^n Y_i\left\{[K_z\ast P_0\mu](Z_j) - P_0\mu(Z_j)\right\}\bar K\left(h_n^{-1}(W_i-W_j)\right)
		\end{aligned}
	\end{equation*}
	with $[K_z\ast P_0\mu](z) = h_n^{-p}\int K\left(h_z^{-1}(z-u)\right)P_0\mu(u)\dx u$. Under Assumption~\ref{as:degenerate_npiv}, the strong law of large numbers gives
	\begin{equation*}
		\zeta_n\xrightarrow{a.s} \E\left[YP_0\mu(Z)\right]h_w^{-q}\bar K(0).
	\end{equation*}
	Since $\E[\phi(Z)|W]=0,\forall \phi\in L_{2,0}(Z)$, $\mathbf{U}_n$ is a centered degenerate {U-statistic}. By the central limit theorem for degenerate U-statistics, see \cite{gregory1977large},
	\begin{equation*}
		\mathbf{U}_n = \frac{2}{n}\sum_{i<j}h(X_i,X_j) \xrightarrow{d} \sum_{j=1}^\infty \lambda_j(\chi^2_{j} - 1).
	\end{equation*}
	Finally, decompose $R_n = R_{1n} + R_{2n}$, where
	\begin{equation*}
		\begin{aligned}
			R_{1n} & = \frac{1}{n}\sum_{i=1}^n Y_i\left\{[K_z\ast P_0\mu](Z_i) - P_0\mu(Z_i)\right\}h_w^{-q}\bar K(0), \\
			R_{2n} & = \frac{1}{n}\sum_{i<j} Y_i\left\{[K_z\ast P_0\mu](Z_j) - P_0\mu(Z_j)\right\}h_w^{-q}\bar K\left(h_w^{-1}(W_i-W_j)\right).
		\end{aligned}
	\end{equation*}
	Note that
	\begin{equation*}
		\begin{aligned}
			\E|R_{1n}| & \leq \E\left|Y\left\{[K_z\ast P_0\mu](Z) - P_0\mu(Z) \right\}\right| h_w^{-q}\bar K(0) \\
			&\lesssim \int \left|[K_z\ast P_0\mu](z) - P_0\mu(z)\right| f_Z(z)\dx z  \\
			& \leq \|K_z\ast \mu - \mu\|^2\|f_Z\|^2 \\
			& = o(1),
		\end{aligned}
	\end{equation*}
	where the first two lines follow from Assumption~\ref{as:degenerate_npiv} (i)-(ii), the third by the Cauchy-Schwartz inequality and $\|P_0\|\leq 1$, and the last  by \cite{gine2015mathematical}, Proposition 4.1.1. (iii). Similarly, since $\E\left[|Y|^2|W\right]<\infty$ a.s. and $\bar K\in L_\infty$, we use the moment inequality in \cite{korolyuk2013theory}, Theorem 2.1.3, to obtain
	\begin{equation*}
		\begin{aligned}
			\E|R_{2n}|^2 & \lesssim \E\left|Y\left\{[K_z\ast P_0\mu](Z') - P_0\mu(Z')\right\}h_w^{-q}\bar K\left(h_w^{-1}(W-W')\right)\right|^2 \\
			& \lesssim \int \left|[K_z\ast P_0\mu](z) - P_0\mu(z)\right| f_Z(z)\dx z  = o(1).
		\end{aligned}
	\end{equation*}
	This completes the proof.
\end{proof}

\newpage
\setcounter{page}{1}
\setcounter{section}{0}
\setcounter{equation}{0}
\setcounter{table}{0}
\setcounter{figure}{0}
\renewcommand{\theequation}{S.\arabic{equation}}
\renewcommand\thetable{S.\arabic{table}}
\renewcommand\thefigure{S.\arabic{figure}}
\renewcommand\thesection{S.\arabic{section}}
\renewcommand\thesubsection{S.\arabic{section}.\arabic{subsection}}
\renewcommand\thepage{Supplementary Material - \arabic{page}}

\begin{center}
	{\LARGE\textbf{SUPPLEMENTARY MATERIAL}}	
\end{center}
\bigskip

\section{Generalized Inverse}\label{app:gi}
In this section, we collect some facts about the generalized inverse operator from operator theory; see also \cite{carrasco2007linear} for a comprehensive review of different aspects of the theory of ill-posed inverse models in econometrics. Let $\varphi\in\mathcal{E}$ be a structural parameter in a Hilbert space $\mathcal{E}$ and let $K:\mathcal{E}\to\mathcal{H}$ be a bounded linear operator mapping to a Hilbert {space} $\mathcal{H}$. Consider the functional equation
\begin{equation*}
	K\varphi = r.
\end{equation*}
If the operator $K$ is not one-to-one, then the structural parameter $\varphi$ is not point identified, and the identified set is a closed linear manifold described as $\Phi^\mathrm{ID}=\varphi+\mathcal{N}(K)$, where $\mathcal{N}(K)=\{\phi:K\phi = 0\}$ is the null space of $K$; see Figure~\ref{fig:geometry}. The following result offers equivalent characterizations of the identified set; see \cite{groetsch1977generalized}, Theorem 3.1.1 for a formal proof.
\begin{proposition}\label{prop:identified_set}
	The identified set $I_0$ is characterized as a set of solutions to
	\begin{itemize}
		\item[(i)] the least-squares problem: $\min_{\phi\in\mathcal{E}}\|K\phi - r\|$;
		\item[(ii)] the normal equations: $K^*K\phi = K^*r$, where $K^*$ is the adjoint operator of $K$.
	\end{itemize}
\end{proposition}

The generalized inverse is formally defined below.
\begin{definition}
	The generalized inverse of the operator $K$ is a unique linear operator $K^\dagger: \mathcal{R}(K)\oplus\mathcal{R}(K)^\perp \to \mathcal{E}$, defined by $K^\dagger r=\varphi_1$, where $\varphi_1\in I_0$ is a unique solution to
	\begin{equation}\label{eq:minimum_norm}
		\min_{\phi\in I_0}\|\phi\|.
	\end{equation}
\end{definition}
For nonidentified linear models, the generalized inverse maps $r$ to the unique minimal norm element of $I_0$. It follows from equation~(\ref{eq:minimum_norm}) that $\varphi_1$ is a projection of $0$ on the identified set. Therefore, $\varphi_1$ is the projection of the structural parameter $\varphi$ on the orthogonal complement to the null space $\mathcal{N}(K)^\perp$, see Figure~\ref{fig:geometry}, and we call $\varphi_1$ the best approximation to the structural parameter $\varphi$. The generalized inverse operator is typically a discontinuous map, as illustrated in the following proposition; see \cite{groetsch1977generalized}, pp.117-118 for more details.

\begin{figure}
	\centering
	\begin{tikzpicture}[font=\sffamily, >=latex]
\clip (-1.3,-1) rectangle (9,8);

\draw (1.5,0.5) coordinate(d0min) -- (8,4) coordinate(d0max);
\node [right] at (d0min) {$\mathcal{N}(K)^\perp=\left\{\phi:\langle\phi,\psi \rangle = 0,\;\forall\psi\in \mathcal{N}(K)\right\}$};

\draw (3,6) coordinate(s0min) -- (6,1) coordinate(s0max);
\draw ($(s0min)+(1,1)$) coordinate(s1min) -- +($(s0max)-(s0min)$) coordinate(s1max);
\node [left] at (s0min) {$\mathcal{N}(K) = \left\{\phi: K\phi=0\right\}$};
\node [right] at (s1min) {$\Phi^{\mathrm{ID}} = \left\{\phi: K\phi=r\right\}$};

\coordinate (d0s0) at (intersection of d0min--d0max and s0min--s0max);
\coordinate (d0s1) at (intersection of d0min--d0max and s1min--s1max);

\node [below] at (d0s0) {0};

\end{tikzpicture}
	\caption{Fundamental subspaces of $\mathcal{E}$.}
	\label{fig:geometry}
\end{figure}

\begin{proposition}
	Suppose that the operator $K$ is compact. Then the generalized inverse $K^\dagger$ is continuous if and only if $\mathcal{R}(K)$ is finite-dimensional.
\end{proposition}
The following example illustrates this when $K$ is an integral operator on spaces of square-integrable functions.
\begin{example}
	Suppose that $K$ is an integral operator
	\begin{equation*}
		\begin{aligned}
			K: L_2 & \to L_2 \\
			\phi & \mapsto \int \phi(z)k(z,w)\dx z.
		\end{aligned}
	\end{equation*}
	Then $K$ is compact whenever the kernel function $k$ is square integrable. In this case, the generalized inverse is continuous if and only if $k$ is a degenerate kernel function
	\begin{equation*}
		k(z,w) = \sum_{j=1}^m\phi_j(z)\psi_j(w).
	\end{equation*}
\end{example}

It is worth stressing that in the NPIV model, the kernel function $k$ is typically a non-degenerate probability density function. Moreover, in econometric applications, $r$ is usually estimated from the data, so that $K^\dagger\hat r\xrightarrow{p} K^\dagger r=\varphi_1$ may not hold even when $\hat r\xrightarrow{p} r$ due to the discontinuity of $K^\dagger$.\footnote{In practice, the situation is even more complex because the operator $K$ is also estimated from the data.} In other words, we are faced with an ill-posed inverse problem. Tikhonov regularization can be understood as a method that smooths out the discontinuities of the generalized inverse $(K^*K)^\dagger$.\footnote{By Proposition~\ref{prop:identified_set}, solving $K\varphi = r$ is equivalent to solving $K^*K\varphi = K^*r$. The latter is more attractive to work with because the spectral theory of self-adjoint operators in Hilbert spaces applies to $K^*K$.}

\section{Degenerate U-statistics in Hilbert Spaces}\label{app:wiener}
\subsection{Wiener-It\^{o} Integral}
This section reviews key results on the asymptotic distribution of degenerate U-statistics in Hilbert spaces. Let $(\mathcal{X},\Sigma,\mu)$ be a measure space, and let $H$ be a separable Hilbert space. We denote by $L_2(\mathcal{X}^m,H)$ the space of functions $f:\mathcal{X}^m\to H$ satisfying $\E\|f(X_1,\dots,X_m)\|^2<\infty$. A stochastic process $\left\{\mathbb{W}(A),A\in\Sigma_\mu\right\}$, indexed by the $\sigma$-field $\Sigma_\mu=\{A\in\Sigma:\; \mu(A)<\infty\}$ is called a \textit{Gaussian random measure} if:
\begin{enumerate}
	\item For all $A\in\Sigma_\mu$,
	\begin{equation*}
		\mathbb{W}(A) \sim N(0,\mu(A)).
	\end{equation*}
	\item For any collection of disjoint sets $(A_k)_{k=1}^K$ in $\Sigma_\mu$, the random variables $\mathbb{W}(A_k),k=1,\dots, K$ are independent, and
	\begin{equation*}
		\mathbb{W}\left(\bigcup_{k=1}^KA_k\right) = \sum_{k=1}^K\mathbb{W}(A_k).
	\end{equation*}
\end{enumerate}

Now, let $(A_k)_{k=1}^K$ be pairwise disjoint sets in $\Sigma_\mu$, and consider the set $S_m$ of simple functions $f\in L_2(\mathcal{X}^m,H)$ of the form
\begin{equation*}
	f(x_1,\dots,x_m) = \sum_{i_1,\dots,i_m=1}^Kc_{i_1,\dots,i_m}\one_{A_{i_1}}(x_1)\times\dots\times \one_{A_{i_m}}(x_m),
\end{equation*}
where $c_{i_1,\dots,i_m}=0$ if any two indices $i_1,\dots,i_m$ are equal, i.e., $f$ vanishes on the diagonal. For a Gaussian random measure $\mathbb{W}$ corresponding to $P$, we define the random operator $J_m:S_m\to H$ by
\begin{equation*}
	J_m(f) = \sum_{i_1,\dots,i_m=1}^Kc_{i_1,\dots,i_m}\mathbb{W}(A_{i_1})\dots\mathbb{W}(A_{i_m}).
\end{equation*}
The operator $J_m$ has three notable properties:
\begin{enumerate}
	\item Linearity;
	\item $\E J_m(f) = 0$;
	\item Isometry: $\E\langle J_m(f),J_m(g)\rangle_H = \langle f,g \rangle_{L_2(\mathcal{X}^m,H)}$.
\end{enumerate}
Since $S_m$ is dense in $L_2(\mathcal{X}^m,H)$, $J_m$ can be extended to a continuous linear isometry on $L_2(\mathcal{X}^m,H)$, known as the Wiener-It\^{o} integral.

\begin{example}
	Let $(B_t)_{t\geq 0}$ denote a real-valued Brownian motion. For any interval $(t,s]\subset [0,\infty)$, define $\mathbb{W}((t,s]) = B_s - B_t$, which is a Gaussian random measure (with $\mu$ as the Lebesgue measure). The Wiener-It\^{o} integral $J:L_2([0,\infty),\dx t)\to \R$ is then given by $J(f) = \int f(t)\dx B_t$.
\end{example}

\subsection{Central Limit Theorem}
Consider a probability space $(\mathcal{X},\Sigma,P)$, where $\mathcal{X}$ is a separable metric space and $\Sigma$ is a Borel $\sigma$-algebra. Let $(X_i)_{i=1}^n$ be i.i.d. random variables taking values in $(\mathcal{X},\Sigma,P)$. Define a symmetric function $h:\mathcal{X}\times\mathcal{X}\to H$, where $H$ is a separable Hilbert space. The $H$-valued {$U$-statistic} of degree $2$ is given by
\begin{equation*}
	\mathbf{U}_n = \frac{2}{n(n-1)}\sum_{1\leq i<j\leq n}h(X_i,X_j).
\end{equation*}
The {$U$-statistic} is called degenerate if $\E h(x_1,X_2)=0$. The following theorem provides the limiting distribution of degenerate $H$-valued $U$-statistics; see \cite{korolyuk2013theory}, Theorem 4.10.2 for a detailed proof.

\begin{theorem}\label{thm:borovskikh}
	Suppose that $\mathbf{U}_n$ is a degenerate {$U$-statistic} such that $\E h(X_1,X_2) = 0$ and $\E\|h(X_1,X_2)\|^2<\infty$. Then
	\begin{equation*}
		n\mathbf{U}_n \xrightarrow{d}J(h),
	\end{equation*}
	where $J(h)=\iint_{\mathcal{X}\times\mathcal{X}} h(x_1,x_2)\mathbb{W}(\dx x_1)\mathbb{W}(\dx x_2)$ is a stochastic Wiener-It\^{o} integral, and $\mathbb{W}$ is a Gaussian random measure on $H$.
\end{theorem}

\end{document}